\documentclass[11pt]{article}
\usepackage[linktocpage,
  pdfpagelabels=true,pdfstartview=Fit,bookmarksopen=true 
  ,pdfpagemode=UseOutlines,bookmarksnumbered=true,
]{hyperref} % make clickable links in LaTeX
\usepackage{bookmark}

\usepackage[a4paper,width=150mm,top=25mm,bottom=25mm,bindingoffset=6mm]{geometry}

\usepackage{multirow} 
\usepackage[table,xcdraw]{xcolor} 
\usepackage{colortbl,hhline} 
\usepackage{diagbox} 

\usepackage{amsmath}
\usepackage[utf8]{} 
\usepackage[T1]{fontenc}
\usepackage[greek,english]{babel}

\usepackage{authblk} 
\usepackage{graphicx} 

\title{Single vs Dynamic Lead-Time Quotations in Make-To-Order Systems with Delay-Averse Customers}

\author[1]{Myron Benioudakis}
\affil[1]{{\footnotesize  Department of Management Science and Technology\\
  Athens University of Economics and Business\\
  Athens, Greece\\
  \url{benioudakis@aueb.gr}, ORCID:0000-0003-3295-3343}}
  
\author[2]{Apostolos Burnetas}
\affil[2]{{\footnotesize  Department of Mathematics\\ National and Kapodistrian University of Athens\\
Athens, Greece\\
\url{aburnetas@math.uoa.gr}, ORCID:0000-0002-9365-9255}}  

\author[3]{George Ioannou} 
\affil[3]{{\footnotesize  Department of Management Science and Technology\\
  Athens University of Economics and Business\\
  Athens, Greece\\
  \url{ioannou@aueb.gr}}}

\date{September 4, 2020}
\usepackage{amssymb} 
\usepackage{amsthm}

\usepackage{fancyhdr}
\usepackage{breqn} 

\usepackage{natbib}
\providecommand{\keywords}[1]{\textbf{\textit{Keywords:}} #1}

\newtheorem{lemma}{Lemma}
\newtheorem{proposition}{Proposition}
\newtheorem{theorem}{Theorem}

\usepackage{mathtools}

\DeclarePairedDelimiter\floor{\lfloor}{\rfloor}

\begin{document}
\maketitle
\begin{abstract}
We develop a model for lead-time quotation in a Markovian Make-To-Order production or service system with strategic customers who exhibit risk aversion. Based on a CARA utility function of their net benefit, customers make individual decisions to join the system or balk by observing the state of the queue. The decisions of arriving customers result in a symmetric join/balk game. Regarding the firm's strategy, the provider announces a lead-time quotation for each state and a respective balking threshold. There is also a fixed entrance fee and compensation rate for the part of a customer' delay exceeding the quoted lead-time. Moreover, we consider the problem from the point of view of a social optimizer who maximizes the total net benefit of the system.

We analyze the provider’s and social optimizer’s maximization problems and we consider two cases regarding the class of lead-time quotation policies, i.e., dynamic and single. We identify the optimal entrance thresholds in each case. Finally, through computational experiments we quantify the effect of risk aversion on the profits and the degree of flexibility that the compensation policy offers. It is shown that the detrimental effects of risk aversion can be addressed more efficiently for the provider's problem compared to the social optimizer's one. Furthermore, the profit loss when setting a single lead-time quote is generally small compared to the optimal dynamic quotation policy.

\end{abstract}

\section{Introduction}
%General intro things (done)
Nowadays, more and more firms adopt Make-To-Order (MTO) strategies to produce products or provide a service, motivated by the competitive advantages offered by the reduced inventory costs. On the other hand, the lack of stock leads to congestion which in turn creates queues. The consequences of congestion are exemplified as the customers' aversion to delay increases, for example when customers require immediate or urgent service (necessary equipment, healthcare etc.). As a result the 
competitive advantage earned by the reduced inventory may be dampened or lost, when customers are very sensitive to delay. The firm may then have to design specific policies to effectively address the customers' concerns about delays so that the demand is not adversely affected. Such policies include dynamic adjustments of the production or service capacity based on the system congestion, flexible pricing strategies, providing customers with information about system congestion, etc.

In this paper we focus on a class of policies that combine lead-time quotations to incoming customers with some degree of compensation for excess delays beyond the quoted lead-time. We show that the advantage of such policies with respect to dynamic pricing is that they allow the producer or the service provider to avoid uniform price reductions and instead compensate only those customers who suffer long delays. We consider these policies in an observable framework where customers are informed about system congestion and study how the quotation/compensation mechanism interacts with the sharing of congestion information.

In terms of customers and demand modeling, in addition to the strategic aspect of customer behavior, we adopt a general framework where customers may exhibit nonlinear sensitivity to increasing delays. We model this essentially risk-averse behavior by a suitably parametrized utility function that allows us to explore the effect of the customers' delay preferences on the provider's quotation policy.

We adopt three viewpoints in analysis the interactions among these factors: the individual customer who responds to the congestion information and the offered quotation by deciding whether to join the system or balk, the service provider who designs a quotation policy with the objective of maximizing his/her net profit, and a central planner, referred to as the social optimizer who considers the total benefit of all incoming customers as well as the provider.

Several questions arise in this context. First and foremost, regarding the structure of the optimal quotation policy, and in particular the comparison between a simple static quotation where the provider offers the same lead-time to all incoming customers and a more complex but potentially more profitable dynamic policy where the quote is also dependent, in addition to the customer's reaction. The question is to what extent the added complexity of the dynamic policies offers an analogous improvement in the provider's profit or the social welfare, as well as, how these comparisons depend on the degree of customer delay sensitivity. 

To address these questions, we develop a mathematical model for lead-time quotation. In particular, we consider a Make-To-Order or a service system where customers arrive according to a Poisson process, and the production/service times are i.i.d. exponential random variables. This gives rise to an M/M/1 queuing model. Customers place a value on the service they receive and incur a cost per unit of time of delay. They also pay an entrance fee and receive a delay compensation for the part of their delay which exceeds the quoted lead-time. The risk aversion is modeled by a utility function of the net customer benefit, that belongs to the class of Constant Absolute Risk Aversion (CARA). Based on this utility customers decide individually whether to join the system or not. Customers observe the actual state of the queue before they make their decision to join the system or balk. 

The fact that customers have full information on system congestion upon arrival, together with the FCFS service discipline reduces the need to take into account the behavior of other customers when making a join/balk decision, thus simplifying the analysis of equilibrium strategies. On the other hand, the service provider's pricing/compensation policies become more complex, since they incorporate the dynamic aspect of congestion-based quotations. 

The provider and the social optimizer employ state-dependent or single quotation policies. In the case of dynamic policies, there is more flexibility for the service provider which leads to higher profits. However, they are more complicated to apply, may require more expensive equipment, and in some cases can cause customer complaints for unequal treatment. On the other hand, by setting a single lead-time, the provider has a simpler process which however is generally less profitable.

The main contribution of this paper is in identifying the provider's/social optimizer's optimal quotation policies, analyzing the effect of customer risk aversion on the provider's/social optimizer's optimal profits and in comparing the single and the dynamic lead-time cases. Specifically, we show that the customer behavior is characterized by an appropriate balking threshold in all problems. We formulate four two-stage maximization problems, for the provider and social optimizer, under dynamic or single quotes, to find the optimal lead-time quotation strategy in each case and the corresponding balking threshold which is bounded in a finite interval. To solve each problem, we first solve a sub-problem where the threshold is fixed and the optimal lead-time quotation strategy is derived. Given this solution, we solve the initial problem by finding the optimal thresholds. For the dynamic case we present algorithms that simplify the computations, while for the single case we find the optimal threshold by a search method in the interval which is determined in an appropriate interval. 

Finally, through computational experiments, we explore the effect of risk aversion on balking thresholds and profits and the degree of flexibility that the presence of compensation offers. We obtain several interesting insights. We see that the provider's maximization may lead to longer queues than is socially optimal, which is in opposite direction to standard results in the case of no compensation (e.g., \cite{naor1969}). Moreover, the profits are mainly affected by the variation of the entrance fee and the increase of compensation rate offers more flexibility to deal with higher degrees of aversion and simultaneously increase the optimal balking thresholds and profits, while at the same time adopting a higher entrance fee. It is also shown that the service provider can address the effect of risk aversion more efficiently than the social optimizer and the optimal dynamic quotes lead to a rather small profit advantage compared to optimal single one. We finally examine the sensitivity of the minimum required capacity in the level of risk aversion.

The paper is organized as follows. Section \ref{literature} provides a literature review. In Section \ref{modeldescription} we define the model and the four optimization problems under risk aversion, i.e, provider and social maximization problems when the lead-time quotation strategy is dynamic or single. In Section \ref{Profit_maximization} we solve the problems under the scope of the provider while in Section \ref{Social_welfare} from the point of view of the social optimizer. In Section \ref{effectsofriskaversion} we derive some analytical properties, regarding the effect of risk aversion compared to the risk-neutral case. In Section \ref{Numericalsection} we present some computational results and quantify the effect of risk aversion. Finally, conclusions and possible extensions are presented in Section \ref{Conclusions}.

\section{Literature Review}\label{literature}
Many aspects of the strategic customer behavior in a queuing system have been studied in recent years.  For comprehensive reviews of the literature we refer to \cite{hassin2003queue} and \cite{Hassin2016}. In this paper, we consider an observable problem and focus on the interaction between customer risk preferences and lead-time quotation/threshold policies, as well as their effects on the service provider and social optimizer profit. 

Models where customers exhibit nonlinear aversion to delay have been developed in observable and unobservable settings. 
\cite{Sun2012} consider the joining decisions of customers with non-linear waiting cost functions. They derive equilibrium thresholds and provide numerical examples where the equilibrium, social, and profit-maximizing thresholds are in decreasing order as in \cite{naor1969}. \cite{Hassin2017RAOBS} examine a general observable G/G/s model using a non-linear utility function for customers. They examine the relationship between the optimal social threshold and monopoly threshold, and they prove their relationship does not generally agree with Naor, but depends on the structure of customer value of joining. They also examine an abandonment model where each customer may leave the system within some amount of time.
\cite{Feng2017} consider a Markov Decision Process model with multiple distinguishable customer types who receive a compensation for the delay above the quoted lead-time. The risk preferences are modeled by a boundary valuation curve between entrance fee and lead-time which is differentiated by customer type, and optimal dynamic pricing and lead-time quotations policies are derived.

\cite{Afeche2013} analyze an unobservable system with multiple customer types, and consider general pricing policies that depend on the actual sojourn time. When the customer types are known, it is shown that the policy that provides full delay compensation and charges an entrance fee equal to the customer's service value maximizes profits. When the types are not distinguishable, incentive compatible policies are constructed based on linear pricing schemes.

In another unobservable setting, \cite{Benioudakis19} explore pricing compensation policies that guarantee a particular demand pattern over time which aim to increase the reputation of the product/company (load control problem). 
In contrast to this setting, the present work considers the problem of equilibrium customer behavior and setting optimal lead-time quotes in an observable framework, thus both customer and service provider strategies are in principle state-dependent. One central question considered in this framework is to what extent the dynamic aspect of the quotation strategies affects the service provider's profit and the total customer benefit. 

Both the unobservable and the observable settings are considered in \cite{Wang2018}. The risk preferences (risk-averse or risk-seeking) are modeled by a quadratic service utility function that involves the mean and the variance of the waiting time. It is shown that when customers are highly risk-averse, then providing the queue length information adversely affects the service provider profits, but it is beneficial for the social benefit.

Regarding dynamic pricing policies, \cite{Duenyas1995a} and \cite{Duenyas1995b} study dynamic lead-time quotation policies in an M/M/1 queue in single and multiclass settings, and show that the optimal quotation is increasing in the queue length. 
\cite{Ata2006} consider dynamic control of the arrival and service rate in an M/M/1 queue to maximize long-run average profit and they solve an associated dynamic pricing problem.
\cite{Celik2008} model a Make-To-Order firm which offers a menu of dynamically adjusted price/lead-time pairs. They show that pricing decisions depend on the aggregate system workload, instead of the separate queue lengths and that the order closest to violating the lead-time must be given priority. 
\cite{Savasaneril2010} develop a dynamic lead-time quotation model of a base-stock M/M/1 inventory system. There is a lateness cost for the provider when customers receive their order after the  quoted lead-time. They formulate the problem as a Markov Decision Process where customers enter the system with a decreasing probability function of the lead-time. Given a base-stock level they determine the optimal lead-time quotation policy and afterward the optimal base-stock level.
\cite{Feng2011} model a G/M/1 queue with heterogeneous service valuations for dynamic pricing and lead-time quotation using Semi-Markov Decision Process. They derive a threshold type structure of the optimal policy.
\cite{Zhao2012} consider a Make-To-Order system where customers belong to a price-sensitive or a lead-time-sensitive class. They compare uniform quotation strategies of a single price and lead-time quotation as well as menu-based strategies with several price and lead-time quote combinations. 
\cite{Hafzoglu2016} consider a model with two priority classes, with static and dynamic price/quotation pairs for the higher and lower priority class respectively. They explore the benefit of the dynamic policy and analyze the problem of optimal mix of two customer types. 
\cite{Oner-Kozen2018} combine a dynamic price/lead-time quotation problem with due-date-based sequencing to two types of customers. A key insight is that a joint decision approach allows the provider to set price/lead-time quotes which increase the acceptance rate and the service level.

For the more general question of static vs dynamic pricing, \cite{Paschalidis2000} consider a queue with multiple customer classes differentiated according to price sensitivity. They show that the benefit of dynamic pricing diminishes as the service capacity increases.
\cite{Haviv2014} model a M/M/1 queue with price and delay-sensitive customers. It is shown that the performance of the demand-independent pricing is significant for both the profit and the social optimizer, compared with the optimal static, demand-dependent price.
\cite{Wang2019} model a two-station tandem queue where the service provider offers dynamic or static prices to price-sensitive customers, using an MDP model with discounted and average criteria award.

The problem of dynamic vs static pricing has also been studied extensively, in settings where there is no customer delay considerations. In an inventory setting, \cite{Gallego1994} deal with this question using intensity control theory for exponential and general demand functions. They show that by using a static price, the firm can earn profits close to those under the optimal dynamic policy. In a similar manner, \cite{Chen2018} analyze four models, three for dynamic pricing and one with a static price, and investigate how the number of price adjustments affects the profitability of deteriorating products. They also explore the impact of menu costs on dynamic pricing decisions. \cite{Liu2019} analyze a two-period model between a firm and a pool of strategic customers. They find that a pre-announced price is more preferable than dynamic pricing as customers become more strategic.

\section{Model Description}\label{modeldescription}
We model a Make-To-Order (MTO) system or a service where customers are identical and place orders or service requests one at a time. Customers arrive according to a Poisson process with rate $\lambda$, and the service times are exponentially distributed with rate $\mu$. The system operates under the First-Come-First-Served (FCFS) discipline. The system under consideration is an M/M/1 MTO queue. When an order is completed, it brings a revenue of $R$ to the customer. There is an entrance fee $p$ as well as a waiting cost $c$ per unit of sojourn time in the system.
The service provider quotes a lead-time $d$ for the service completion to an incoming customer. If the actual time $X$ that the customer stays in the system is longer than the quoted lead-time then there is a compensation $l$ per unit time of lateness $(X-d)^{+}$. Moreover, the state of the system is observable to all parties. Based on that information ($R,c,p,\lambda, \mu, l, d$), as well as the system length $n$, arriving customers decide whether to join or balk. If they balk they choose an alternative option with revenue $v$. Without loss of generality we assume that $v=0$. The quoted lead-time as well as each customer's decision generally depend on $n$. 

We also assume that $p\leq R, l \leq c$. These assumptions are consistent with pricing in a risk-averse customer framework. If the provider considered entrance fees $p>R$, then he or she should have to also use $l>c$ in order to entice the customers to join. Under such a policy customers would pay a high price to join the system, betting on long delays and subsequent high compensations. This framework is beyond the scope of our analysis.

The net benefit from joining is equal to $R-p-cX+l(X-d)^{+}$. We also assume that customers are risk-averse. Risk aversion is modeled by a concave utility function of the net benefit. 
The utility function of a risk-averse customer who faces a delay $X$ is defined as: $U(X)=\frac{1-e^{-r(R-p-cX+l(X-d)^{+})}}{r}$ with $r>0$, which is a utility function in the class of Constant Absolute Risk Aversion (CARA), with absolute risk aversion coefficient $r$. Note that $U$ is decreasing in $r$ for all
$X$, and as $r$ diminishes to zero the utility function converges to the risk-neutral form $U(X)=X$.

The delay of a customer who joins when $n$ or more customers are present is a random variable $X_n\sim Gamma(n+1,\mu)$, with probability density function $f_{n,\mu}(x)=\frac{\mu^{n+1}  x^{n} e^{-\mu x}} {n!}$.  Customers maximize their expected utility: 
\begin{equation*}
B_n(d)=E(U(X_n))=\displaystyle \int_0^d \frac{1-e^{-r(R-p-cx)}}{r} f_{n,\mu}(x)\,dx +\int_d^\infty \frac{1-e^{-r(R-p-ld-(c-l)x)}}{r} f_{n,\mu}(x)\,dx.
\end{equation*}
After computations it follows that, when $\mu > r(c-l)$:
\begin{equation}\label{SumBn}
B_n(d)=\frac{1-e^{-r(R-p)}\left(\left(\frac{\mu}{\mu-rc}\right)^{n+1}\left(1-K_1(n)\right)+\left(\frac{\mu}{\mu-r(c-l)}\right)^{n+1} K_2(n)\right)}{r} 
\end{equation}
where $\displaystyle K_1(n)=\sum_{k=0}^{n} \frac{e^{-(\mu-rc)d} \left((\mu-rc)d\right)^{k}}{k!} \text{ and } K_2(n)=\sum_{k=0}^{n} \frac{e^{-(\mu-rc)d} \left((\mu-r(c-l))d\right)^{k}}{k!}.$ 

On the other hand, if $\mu \leq r(c-l)$ the expected utility $B_n(d)=- \infty$ for all $n \geq 0,d\geq 0$.
 i.e., when the service capacity is not sufficiently large to compensate for the waiting cost and the degree of aversion, no customer will ever enter the system, regardless of the value of service $R$. 
 
This property is in pronounced contrast with the risk-neutral customer model. Indeed, when $r \to 0$ then the expected delay is finite for all $n \geq 0$ and $\mu>0$, even with very low capacity. The customer's optimal response is to join the system if $B_n(d)\geq0$ and balk otherwise. The utility of balking is $0$.

The service provider employs a lead-time quotation strategy $D=(d_0,d_1, d_2,\ldots)$ to maximize his/her expected profit per unit time. We may assume that the provider is risk-neutral as long as a large number of risk-averse customers enter the system and each of their service is only a small part of the total profit (see \cite{gamesofstrategy}, p. 277). Let
\begin{equation*}
G_n(d) = \begin{cases}
E(p-l(X_n-d)^{+})=p-lL_n(d),& \text{if }  B_n(d) \geq 0\\
0,&   \text{if }  B_n(d) < 0\\
\end{cases}
\end{equation*}
be the service provider's expected net profit from a customer who joins when $n$ customers are in the system, where $\displaystyle L_n(d)=E(X-d)^{+}= \int\limits_{d}^\infty (x-d)\frac{\mu^{n+1}}{n!}x^{n}e^{-\mu x}\mathrm{d}x$ is the expected delay from the lead-time quotation. The provider's objective is to maximize the expected profit per unit time in infinite horizon. 

In the next Lemma we show that the expected utility function of a customer is decreasing in $d$ and the expected provider's profit from an entering customer is increasing in $d$.
\begin{lemma}\label{monotonicity}
\begin{itemize}
\item[(i.)] $B_n(d)\text{ decreasing in } n \text{ and decreasing in } d$
\item[(ii.)] $L_n(d)\text{ increasing in } n \text{ and decreasing in } d$
\item[(iii.)] $G_n(d)\text{ decreasing in } n \text{ and increasing in } d$
\end{itemize}
\end{lemma}

We next consider the queueing process under a general quotation policy. We show that under any $D$, there exists a finite queue length above which an incoming customer balks. This implies that in steady state the queue length behaves as an $M/M/1$ queue with finite capacity. Specifically, given a $D$, let $n_0(D)=\min( n\geq 0:B_n(d_n)<0)$, which is the first system state where a customer is induced by the lead-time to balk. In the next proposition we develop lower and upper bounds for $n_0(D)$.

\begin{proposition}\label{range n0}
If $l<c$, for any quotation policy $D$ the corresponding balking threshold $n_0(D)$ satisfies $$\underline{n}\leq n_0(D) \leq \overline{n},$$
where $$\displaystyle \underline{n}=\inf( n:\lim_{d \to \infty} B_n(d) <0 )=\floor* {\frac{r(R-p)}{\ln\left(\frac{\mu}{\mu-rc}\right)}},$$
$$\overline{n}=\inf( n:B_n(0) <0 )=\floor*{\frac{r(R-p)}{\ln\left(\frac{\mu}{\mu-r(c-l)}\right)}}.$$
\end{proposition}

From Proposition \ref{range n0}, we see that when $l=0$, $\underline{n}$ and $\overline{n}$ coincide, therefore, there is a unique threshold $n_0(D)=\floor* {\frac{r(R-p)}{\log\left(\frac{\mu}{\mu-rc}\right)}}$ for any $D$. This is equal to the threshold in \cite{naor1969} for risk-neutral customers. Moreover, if $\underline{n}=0$ then for all $n \geq 0:\lim_{d \to \infty} B_n(d) < 0$ which means that even if the system is empty, the provider must offer a finite lead-time to induce a customer to join. On the other hand, if $\overline{n}=0$ then for all $n \geq 0:B_n(0) < 0$ which means that even if the system is empty and $d=0$, a customer never joins. Therefore, in analogy with \cite{naor1969}, a necessary and sufficient condition so that some customers join the system in equilibrium is that $\overline{n}>0$, or equivalently $R-p \geq \frac{1}{r} \ln\left( \frac{\mu}{\mu-r(c-l)}\right)$. When $l=c$, the upper bound $\overline{n}=\infty$, since $B_n(0)\geq 0$ for all $n$, and the above condition is always satisfied as long as $R-p \geq 0$.

Based on Proposition \ref{range n0}, under any strategy $D$ such that $n_0(D)=n_0$, the system behaves as a CTMC with a single recurrent class $\lbrace0,1,\ldots,n_0\rbrace$ and all states $n > n_0$ are transient. This system corresponds to a $M/M/1/n_0$ with steady-state distribution:
$$
q(n;n_0)=\frac{\rho^n}{\displaystyle\sum_{n=0}^{n_0} \rho^{n}},
n=0,\ldots,n_0,
$$
where $\rho=\frac{\lambda}{\mu}.$

We now turn to the provider's optimization problem. We consider two cases regarding the class of quotation policies. 
First dynamic quotation policies, where the lead-time is allowed to depend on the state, and second, policies with single lead-time, identical for all $n$.

In the first case, there is more flexibility for the service provider and this policy allows him to earn higher profits. However, this policy is more complicated to apply, and in some cases it could cause customer complaints for unequal treatment.

On the other hand, by setting a single lead-time for all states, the provider is able to use a simpler quotation policy. As will be shown below, customers who join in states under the balking threshold benefit from a single lead-time policy.

For the class of dynamic policies the provider's problem is:
$$P^{*}=\sup_D P(D) ,$$
where $\displaystyle P(D)=\lambda \sum_{n=0}^{n_0(D)-1} q(n;n_0(D)) G_n(d_n)$
is the expected net profit per unit time.

The profit maximization can be expressed as a two-stage optimization problem as follows: 
\begin{eqnarray}\label{provider_max}
P^{*}=\max_{\underline{n}\leq n_0 \leq \overline{n}} H(n_0), 
\end{eqnarray}
where $\displaystyle H(n_0)= \sup_{D}(P(D):n_0(D)=n_0)$ is the provider's optimal profit under all policies that induce a balking threshold $n_0$.

In the case of dynamic quotation policies, since we consider steady-state criteria, in order to define a quotation policy, it is sufficient to determine a balking threshold $n_0$, and a vector $D_{n_0}=(d_0, d_1, \ldots,d_{n_0})$ of lead-times quotes which ensure that a customer joins if and only if $n<n_0$. From now on a provider's strategy will be denoted by $(n_0, D_{n_0})$, where $n_0$ is an integer and $D_{n_0}$ a nonnegative vector of dimension $n_0+1$. This does not mean that the provider announces $n_0$, but simply that his/her announced sequence $D$ is such that $n_0$ is the balking threshold. To ensure this, the lead-times $d_0, d_1, \ldots,d_{n_0}$ must be set to satisfy the following constraints:
\begin{align}\label{constraints}
B_n(d_n)&\geq0 ,\quad n=0,1,\ldots,n_0-1\nonumber \\ 
B_{n_0}(d_{n_0})&<0 \\ 
d_n &\geq 0, \quad n=0, 1,\ldots, n_0\nonumber
\end{align}
while $d_{n_0+1},d_{n_0+2}, \ldots$ are irrelevant.

The service provider's optimal quotation strategy given $n_0$ is derived from:
\begin{equation}\label{firstmaxpro}
H(n_0)=\sup_{D_{n_0}} P(n_0,D_{n_0})
\end{equation}
such that the constraints in \eqref{constraints} are satisfied.

In the case of a single lead-time for all states, the problem is identical, with the additional requirement that:
\begin{align}\label{constraintconstant}
d_n=d,\quad n=0,1,\ldots, n_0.
\end{align}
The provider's strategy will be denoted by $(n_0, d)$, where $d$ is the single lead-time for any $n$. In this case, the announced lead-time $d$ must ensure that the balking threshold is $n_0$. Similarly to the previous case the provider's problem is:
$$P^{*}_{c}=\sup_{d} P_{c}(d),$$ 
where $\displaystyle P_c(d)=\lambda \sum_{n=0}^{n_0(d)-1} q(n;n_0(d)) G_n(d)$.

It can also be expressed as a two-stage optimization problem as follows:
\begin{eqnarray}\label{provider_maxc}
P^{*}_{c}=\max_{\underline{n}\leq n_0 \leq \overline{n}} H_{c}(n_0), 
\end{eqnarray}
where $\displaystyle H_{c}(n_0)= \sup_{d}(P_c(d):n_0(d)=n_0)$.

Therefore, the provider's optimal quotation strategy given $n_0$ is derived from:
\begin{equation}\label{firstmaxproc}
H_{c}(n_0)=\sup_{d} P_{c}(n_0,d)
\end{equation}
such that the constraints in \eqref{constraints} and \eqref{constraintconstant} are satisfied.

Regarding the social benefit maximization problem, the only change is in the objective function, since we consider maximization of the total customer and provider benefit per unit time in steady state. Following a similar approach, for the class of dynamic policies we define:
$$S^{*}=\sup_D S(D),$$
where $\displaystyle S(D)=\lambda \sum_{n=0}^{n_0(D)-1} q(n;n_0(D)) ( G_n(d_n)+ B_n(d_n)).$

The two-stage optimization problem is:
\begin{eqnarray}\label{social_max}
S^{*}=\max_{\underline{n}\leq n_0 \leq \overline{n}} Z(n_0), 
\end{eqnarray}
where $\displaystyle Z(n_0)= \sup_{D}( S(D):n_0(D)=n_0 )$.

Therefore, the problem for the case of dynamic lead-time quotations given $n_0$ can be expressed as:
\begin{equation}\label{firstsocpro}
Z(n_0)=\sup_{D_{n_0}} S(n_0,D_{n_0}),
\end{equation}
such that the constraints in \eqref{constraints} are satisfied.

Finally, for the class of single quotation policies the problem is:
$$S^{*}_{c}=\sup_{d} S_{c}(d)$$ 
where $\displaystyle S_c(d)=\lambda \sum_{n=0}^{n_0(d)-1} q(n;n_0(d)) (G_n(d)+B_n(d))$.

The two-stage optimization problem is:
\begin{eqnarray}\label{social_maxc}
S^{*}_{c}=\max_{\underline{n}\leq n_0 \leq \overline{n}} Z_{c}(n_0), 
\end{eqnarray}
where $\displaystyle Z_{c}(n_0)= \sup_{d} (S_{c}(d):n_0(d)=n_0)$.

In a similar manner, the social benefit maximization problem for a single lead-time quotation given $n_0$ can be expressed as:
\begin{equation}\label{firstsocproc}
Z_{c}(n_0)=\sup_{d} S_{c}(n_0,d)
\end{equation}
such that the constraints in \eqref{constraints} and \eqref{constraintconstant} are satisfied.

\section{Profit Maximization-Optimal Quotation Policy}\label{Profit_maximization}
In this section we solve the two-stage problems \eqref{provider_max} and \eqref{provider_maxc} under dynamic and single quotation strategies, respectively. We derive the optimal pair of lead-time quotations and balking threshold that maximizes the service provider's long-run average profit and we present some monotonicity properties.
Specifically in subsection \ref{dynamicprovider} we first determine the optimal dynamic quotation strategy for the problem \eqref{firstmaxpro} given a balking threshold $n_0$ and afterwards we present an algorithm which determines the optimal threshold.
In subsection \ref{singleprovider} we solve problem \eqref{firstmaxproc}. The optimal $n_0$ in this case is computed through an exhaustive search in the corresponding finite interval.

\subsection{Dynamic quotation}\label{dynamicprovider}
To determine the optimal lead-time quotation in problem \eqref{firstmaxpro}, i.e., given a balking threshold $n_0$, we define the sequence:
\begin{equation}\label{tildeoptimalprovider}
\tilde{d}^{P}_{n} = \sup ( d \geq 0: B_n(d) \geq 0 ), \: n\geq0.
\end{equation}

The quantity $\tilde{d}^{P}_{n}$ denotes the maximum possible lead-time quotation that will induce a customer to join in state $n.$ We use the convention $\sup\varnothing=-1$, denoting that if $B_n(d)<0$ for all $d\geq0$, then no customer will join regardless of the lead-time quotation.

Since $B_n(d)$ is continuous and decreasing in $d$, if the set $\lbrace d \geq 0:B_n(d) \geq 0 \rbrace$ is not empty, then it is either a bounded interval $\left[0,\tilde{d}^{P}_{n} \right]$, where $\tilde{d}^{P}_{n}$ is the solution of $B_n(d)=0$, or the half-line $\left[0,\infty \right)$. In both cases $\tilde{d}^{P}_{n}$ can be considered as a well-defined lead-time, with the convention that $\tilde{d}^{P}_{n}= \infty$ corresponds to no compensation. In the next lemma we summarize these. We also show that $\tilde{d}^{P}_{n}$ decreases in $n$ and the expected service provider's profit from a customer who joins when $n$ customers are in the system, decreases in $n$.
\begin{lemma}\label{dn_monotonicity}
\begin{itemize}
\item[(i.)] $\tilde{d}^{P}_{n}$ is decreasing in $n$
\item[(ii.)] $\tilde{d}^{P}_{n}=\infty$ for all $n<\underline{n}$
\item[(iii.)] For any $n \leq n_0-1, G_n(\tilde{d}^{P}_{n})$ is decreasing in $n$
\end{itemize}
\end{lemma}

In the next Proposition we determine the optimal lead-time quotation $D^{P}_{n_0}$. 
\begin{proposition}\label{opt quot}
For any $n_0 \in \left\lbrace \underline{n}, \ldots ,\overline{n} \right\rbrace$, an optimal dynamic quotation strategy in \eqref{firstmaxpro} subject to the condition $n_0(D)=n_0$ is determined as follows:
\begin{equation*}
D^{P}_{n_0}=(\tilde{d}^{P}_{0},\ldots,\tilde{d}^{P}_{n_0-1},\tilde{d}^{P}_{n_0}+\epsilon),
\end{equation*}
for any $\epsilon>0$.
\end{proposition}

As discussed above, the optimal quotation $D^{P}_{n_0}$ is uniquely determined for states $n=0,\ldots ,n_0-1$, while for $n=n_0$, any lead-time $d>\tilde{d}^{P}_{n_0}$ is sufficient.
From Proposition \ref{opt quot} it is also follows that problem in \eqref{firstmaxpro} has an optimal solution in $D$ for any $n_0$ and since $n_0$ is restricted in the finite set $\left\lbrace \underline{n},\underline{n}+1, \ldots ,\overline{n} \right\rbrace$, the supremum in \eqref{firstmaxpro} can be attained.

From sequence \eqref{tildeoptimalprovider} and Proposition \ref{opt quot} it also follows that, although the service provider employs a state-dependent quotation strategy, he or she cannot generally reap the entire benefit from the customers. This is so because in states $n<\underline{n}$ although $\tilde{d}^{P}_{n}=\infty$, i.e. no compensation is offered, the customers may still obtain a strictly positive benefit from joining. To obtain the entire customer benefit the provider must also have pricing flexibility.

From Lemma \ref{dn_monotonicity} and Proposition \ref{opt quot} it follows that the optimal lead-time quote is decreasing in $n$. This is contrast with the results of \cite{Savasaneril2010}, where the optimal quote is increasing in $n$. The difference is due to the fact that in \cite{Savasaneril2010} an incoming customer does not observe the state of the system but instead joins with a prespecified probability $f(d)$, where $d$ is the lead-time. In our model where the customers are strategic and observe $n$, the provider is forced to quote lower lead-times when $n$ is large, so that he or she entices the customer to join.

From Proposition \ref{opt quot} we see that the supremum in \eqref{provider_max} is attained by $D^{P}_{n_0}$. Therefore, 
\begin{eqnarray*}
\displaystyle H(n_0)=P(n_0,D^{P}_{n_0})= \lambda \sum_{n=0}^{n_0-1} q(n;n_0) G_n(\tilde{d}^{P}_{n}).
\end{eqnarray*}
The next step is to determine the optimal threshold value $n_0$ for problem \eqref{provider_max}. The next Theorem provides a necessary and sufficient condition for $n_0$ to be optimal. This result is useful in designing a faster search algorithm for identifying the optimal $n_P$.

\begin{proposition}\label{optimalthresholdprovider}
The optimal $n_P$ for the optimization problem in \eqref{provider_max} is equal to
\begin{equation*}
n_P = \max(\underline{n},\min (\tilde{n}_P,\overline{n})),
\end{equation*}
where, $$\tilde{n}_P=\min ( n_0: A(n_0)<0 ),$$ and
$$A(n_0)=G_{n_0}(\tilde{d}^{P}_{n_0})\displaystyle\sum_{n=0}^{n_0} \rho^{n}-\rho \displaystyle\sum_{n=0}^{n_0-1} \rho^{n} G_n(\tilde{d}^{P}_{n}).$$
\end{proposition}

We can now summarize the results from Propositions \ref{opt quot} and \ref{optimalthresholdprovider} in the following theorem:
\begin{theorem}
The problem of provider profit maximization in the class of dynamic lead-time quotes policies has an optimal solution $(D^{P}_{n_P},n_P)$ and the maximum provider profit is equal to:
\begin{equation*}
P^*=\lambda \sum_{n=0}^{n_P-1} q(n;n_P) G_n(\tilde{d}^{P}_{n}).
\end{equation*}
\end{theorem}

\subsection{Single quotation}\label{singleprovider}
In this subsection we assume that the lead-time quote is the same for all states. We first determine the optimal strategy for \eqref{firstmaxproc}, given a balking threshold $n_0$. Since $B_n(d)$ is decreasing in $n$, constraints \eqref{constraints}, \eqref{constraintconstant} are equivalent to $B_{n_0}(d)<0\leq B_{n_0-1}(d)$. Therefore to maximize profits under $n_0$, the service provider essentially must set a lead-time quote such that for state $n = n_0-1$ the incoming customer weakly prefers to join, i.e., $d^{P_c}_{n_0}=\tilde{d}^{P}_{n_0-1}$.  

\begin{figure}
\centering
\includegraphics[width=12cm,height=7cm]{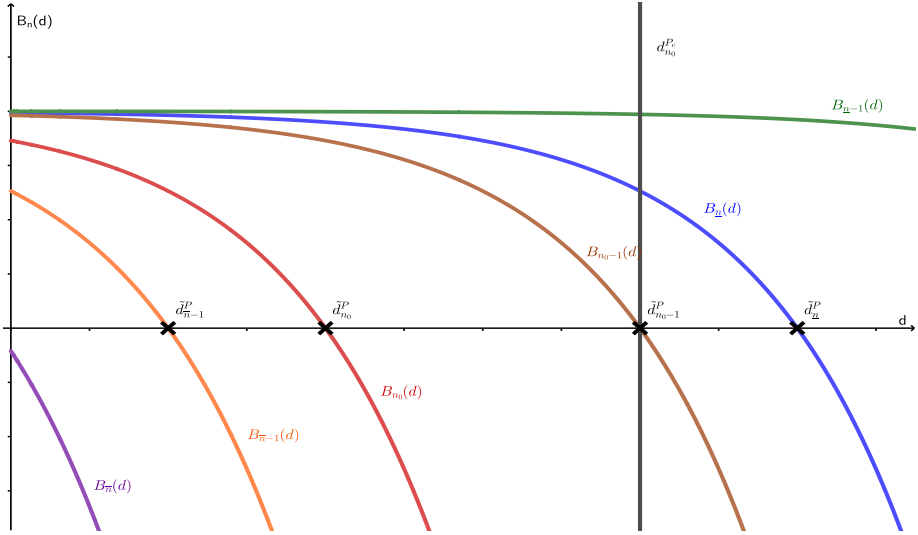}
\caption{Optimal dynamic and single lead-time quotes.}
\label{utilityfun}
\end{figure} 

In Figure \ref{utilityfun} we present several utility function curves for states $n\leq \overline{n}$, to illustrate the difference with the case of dynamic quotes. The corresponding optimal dynamic lead-time quotations for each curve are marked with the black points. The vertical line at $\tilde{d}^{P}_{n_0-1}$ represents the optimal fixed lead-time quotation strategy $d^{P_c}_{n_0}=\tilde{d}^{P}_{n_0-1}$. 

For states $n<n_0-1$ the entering customers have positive utility, in contrast with the dynamic lead-time quotation case, where the provider offers the maximum possible lead-time separately in each state, so that customers weakly prefer to join. Therefore, the single quotation approach is more favorable for the customers. On the other hand, for states $n\geq n_0$ customers have negative utility and they balk. 

Note that for $n_0=\underline{n}$, $d^{P_c}_{\underline{n}}=\tilde{d}^{P}_{\underline{n}-1}=\infty$, i.e., to enforce the lowest possible balking threshold the provider does not offer any compensation. In this case the marginal customer who joins at state $\underline{n}-1$ obtains a generally positive benefit, in contrast to the higher values of $n_0$, where the lead-time is set so that the marginal customer is indifferent. Finally it can easily be seen that the optimal single lead-time quote is decreasing in $n_0$.

We summarize these and we present the optimal single lead-time quotation strategy in the following Proposition:
\begin{proposition}\label{opt quotc}
\begin{itemize}
\item[(i.)]
For any $n_0 \in \left\lbrace \underline{n}, \ldots ,\overline{n} \right\rbrace$, the optimal single lead-time quotation strategy in \eqref{firstmaxproc} subject to the condition $n_0(D)=n_0$ is determined as follows:
\begin{equation*}
d^{P_c}_{n_0} = \begin{cases}
\infty,& \text{if }  n_0 =\underline{n}\\
\tilde{d}^{P}_{n_0-1},&   \text{if }  n_0 \in \left\lbrace \underline{n}+1, \ldots ,\overline{n} \right\rbrace\\
\end{cases}
\end{equation*}

\item[(ii.)] The optimal single lead-time quotation strategy $d^{S_c}_{n_0}$ is decreasing in $n_0$.
\end{itemize}
\end{proposition}

An immediate consequence of Proposition \ref{opt quotc} is that:
\begin{eqnarray*}
\displaystyle H_{c}(n_0)=P_{c}(n_0,d^{P_c}_{n_0})= \lambda \sum_{n=0}^{n_0-1} q(n;n_0) G_n(\tilde{d}^{P}_{n_0-1})
\end{eqnarray*}

The optimal balking threshold $n_{P_c}$ can be determined by an exhaustive search in the finite interval $\left\lbrace \underline{n}, \ldots ,\overline{n} \right\rbrace$.

We can now summarize the results in the following theorem:
\begin{theorem}
The problem of provider profit maximization in the class of single lead-time quotes policies has an optimal solution $(d^{P_c}_{n_0},n_{P_c})$ and the maximum provider profit is equal to:
\begin{equation*}
P_{c}^{*}=\lambda \sum_{n=0}^{n_{P_c}-1} q(n;n_{P_c}) G_n(d^{P_c}_{n_{P_c}}).
\end{equation*}
\end{theorem}

\section{Social Benefit-Optimal Quotation Policy}\label{Social_welfare}
In this section we solve the two-stage problems \eqref{social_max} and \eqref{social_maxc} for optimizing the social benefit per unit time under the class of dynamic and single quotation strategies respectively. The optimization problem is more involved because the provider and the customers have conflicting interests. As a result the optimal quotes are not simply the maximum possible that will induce customers to join as in the profit maximization case, but must be determined by considering the customer benefit as well. 

We follow the same approach as in the profit maximization problem and analyze the problem in two stages, where a balking threshold $n_0$ is set at the first stage and the socially optimal quotation strategy that results in $n_0$ at equilibrium is derived at the second stage.

Under both quotations strategies (dynamic and single) we show that the objective function in the second stage is unimodal in the lead-time, thus the optimal can be found by differentiation. In the first stage, we show that, under dynamic quotation the optimal balking threshold can be determined by an efficient algorithm analogous in Section \ref{dynamicprovider}, whereas under single quotation it is derived through an exhaustive search in the finite interval of balking thresholds $\lbrace \underline{n}, \underline{n}+1, \ldots \overline{n}\rbrace$.

\subsection{Dynamic quotation}\label{dynamicsocial}
In this subsection we solve problems \eqref{social_max} and \eqref{firstsocpro} to find the optimal lead-time quotations and the optimal balking threshold respectively, from the point of view of the social optimizer. The main difference with the corresponding profit maximization model is that when the service provider maximizes his/her profits, he or she sets the lead-time equal to $\tilde{d}^{P}_{n}$ which is the maximum possible lead-time that will entice the customer to join in this state. Here we show that the social optimizer may set a lead-time lower than $\tilde{d}^{P}_{n}$, thus leaving part of the benefit to the customer. We also show that the optimal balking threshold can be determined by an algorithm analogous to that of Proposition \ref{optimalthresholdprovider}. 

We start with the second stage problem. For a given $n_0$, problem \eqref{firstsocpro} can be expressed as: 
\begin{equation*}
Z(n_0)=\sup_{d_0, d_1, \ldots, d_{n_0}} \lambda \sum_{n=0}^{n_0-1} q(n;n_0) ( G_n(d_n)+ B_n(d_n)),
\end{equation*}
where $d_0, d_1, \ldots, d_{n_0}$ are quotes that satisfy \eqref{constraints}.

We first observe that given $n_0$, $d_n$ is constrained in $d_n \leq \tilde{d}^{P}_{n}$ for $n \leq n_0-1$, and $d_{n_0}>\tilde{d}^{P}_{n_0}$.

Furthermore, each term $G_n(d_n)+ B_n(d_n)$ in $Z(n_0)$ can be maximized separately in $d_n$ and the optimal value of $d_n$ does not depend on $n_0$.
Therefore, if we determine the optimal lead-time quotation $\tilde{d}^{S}_{n}$, that induces customers to join in state $n$, for each $n=0,1,\ldots,\overline{n}-1$, we can obtain the optimal quotation strategy given any $n_0\in\lbrace \underline{n}, \underline{n}+1, \ldots \overline{n}\rbrace$ by simply setting the lead-time equal to $\tilde{d}^{S}_{n}$ for $n=0,\ldots,n_0-1$ and any $d_{n_0}>\tilde{d}^{P}_{n}$ for $n=n_0$.

In the following Proposition we derive the optimal lead-time $\tilde{d}^{S}_{n}$ for $n=0,1,\ldots,\overline{n}-1$, after establishing a unimodality property of $G_n(d_n)+B_n(d_n)$.

Let: 
\begin{equation}\label{dtildesoc}
a(d_n)=\frac{\overline{F}_{n,\mu}(d_n)}{\overline{F}_{n,v}(d_n)}e^{-rld_n}-\left( \frac{\mu}{v}\right)^{n+1} e^{-r(R-p)},
\end{equation}
where $v=\mu-r(c-l)$ and $\overline{F}_{n,\mu}(d_n)=P(X_n>d_n)$ for $X_n\sim Gamma(n+1,\mu)$. 

\begin{proposition}\label{unimodal}
For any $n \in \left\lbrace 0,1, \ldots ,\overline{n}-1 \right\rbrace$, $G_n(d_n)+B_n(d_n)$ has a unique maximizing lead-time $\tilde{d}^{S}_{n}<\infty$. Specifically:

If $\tilde{d}^{P}_{n}<\infty$ and $a(\tilde{d}^{P}_{n})\geq 0$ then $\tilde{d}^{S}_{n}=\tilde{d}^{P}_{n}$, otherwise $\tilde{d}^{S}_{n}$ is the unique solution of $a(d_n)=0$ in  $[0,\tilde{d}^{P}_{n})$.
\end{proposition}

The optimal dynamic lead-time quotation strategy for problem \eqref{firstsocpro} is now summarized in the next Proposition:
\begin{proposition}\label{optquotsoc}
For any $n_0 \in \left\lbrace \underline{n}, \ldots ,\overline{n} \right\rbrace$, an optimal dynamic quotation strategy in \eqref{firstsocpro} subject to the condition $n_0(D)=n_0$ is determined as follows:
\begin{equation*}
D^{S}_{n_0}=(\tilde{d}^{S}_{0},\ldots,\tilde{d}^{S}_{n_0-1},\tilde{d}^{P}_{n_0}+\epsilon),
\end{equation*}
for any $\epsilon>0$.
\end{proposition}

As it is expected, the criterion of social benefit maximization is more favorable for the customers. Indeed, since it is generally true that $\tilde{d}^{S}_{n}<\tilde{d}^{P}_{n}$ for states $n \geq \underline{n}$, customers may obtain a positive net benefit from joining, although under profit maximization, their benefit equals zero in these states. Furthermore, even in state $n<\underline{n}$,  $\tilde{d}^{S}_{n}<\tilde{d}^{P}_{n}=\infty$, which means that the social optimizer offers compensation for sufficiently large delays, although under profit maximization the provider does not offer any compensation at all.

From Proposition \ref{optquotsoc} we see that the supremum in \ref{social_max} is attained by $D^{S}_{n_0}$. Therefore, 
\begin{eqnarray*}
\displaystyle Z(n_0)=S(n_0,D^{S}_{n_0})= \lambda \sum_{n=0}^{n_0-1} q(n;n_0) (G_n+B_n)(\tilde{d}^{S}_{n})
\end{eqnarray*}

Turning to the first-stage problem, in Proposition \ref{optimalthresholdsoc} we show that the optimal balking threshold $n_S$ can be determined by an algorithm similar to that in Proposition \ref{optimalthresholdprovider}, instead of a search in all values $\underline{n},\ldots,\overline{n}$.

\begin{proposition}\label{optimalthresholdsoc}
The optimal $n_S$ for the optimization problem in \eqref{social_max} is
\begin{equation*}
n_S = \max(\underline{n}, \min ( \tilde{n}_S,\overline{n} )),
\end{equation*}
where, $$\tilde{n}_S=\min\left( n_0: Y(n_0)<0 \right) ,$$ and
$$Y(n_0)=(G_{n_0}+B_{n_0})(\tilde{d}^{S}_{n_0})\displaystyle\sum_{n=0}^{n_0} \rho^{n}-\rho \displaystyle\sum_{n=0}^{n_0-1} \rho^{n} (G_n+B_n)(\tilde{d}^{S}_{n}).$$
\end{proposition}

We can now summarize the results from Propositions \ref{optquotsoc} and \ref{optimalthresholdsoc} in the following theorem:
\begin{theorem}
The problem of social benefit maximization in the class of dynamic lead-time quotes policies has an optimal solution $(D^{S}_{n_S},n_S)$ and the maximum social benefit is equal to:
\begin{equation*}
S^*=\lambda \sum_{n=0}^{n_S-1} q(n;n_S) (G_n(\tilde{d}^{S}_{n})+B_n(\tilde{d}^{S}_{n})).
\end{equation*}
\end{theorem}

\subsection{Single quotation}\label{singlesocial}
In this subsection we assume that the lead-time quote is the same for all states. We solve problems \eqref{social_maxc} and \eqref{firstsocproc} to find the optimal single lead-time quotation strategy and the optimal balking threshold that maximize the social benefit.
For the second stage problem, the main difference from the case of dynamic quotation is that the separability property of the objective function in $d_n, n=0,\ldots,n_0-1$ does not hold any more, since the quoted lead-time $d$ is the same for all states.

We first find the interval of feasible values for $d$ that ensure a balking threshold $n_0$. We then show that the objective function in \eqref{firstsocproc} is unimodal in $d$ and identify the optimal quote as the solution of an equation.
For the first stage problem the optimal value of the balking threshold $n_0$ is determined by an exhaustive search in the interval $\left\lbrace \underline{n}, \ldots ,\overline{n} \right\rbrace$.

We start with the second stage problem. For a given $n_0$, problem \eqref{firstsocproc} can be expressed as:
\begin{equation*}
Z_c(n_0)=\sup_{d} \lambda \sum_{n=0}^{n_0-1} q(n;n_0) ( G_n(d)+ B_n(d)),
\end{equation*}
where $d$ is a quote that satisfies \eqref{constraints} and\eqref{constraintconstant}.

Let $n_0\in\left\lbrace \underline{n}+1, \ldots ,\overline{n}-1 \right\rbrace$ be a balking threshold.
Since $B_n(d)$ is decreasing in $d$ and $n$, $d$ must be strictly higher than $\tilde{d}^{P}_{n_0}$, so that the $n_0^{\text{th}}$ customer in the queue will be rejected. At the same time, $d$ cannot exceed $\tilde{d}^{P}_{n_0-1}$, otherwise the $(n_{0}-1)^{\text{th}}$ customer will balk. Therefore the optimal lead-time quote $d^{S_c}_{n_0}$ is constrained in the interval $\tilde{d}^{P}_{n_0}<d^{S_c}_{n_0} \leq \tilde{d}^{P}_{n_0-1}$. 
In terms of the optimal single quote under profit maximization, it follows that $d^{P_c}_{n_0+1}<d^{S_c}_{n_0} \leq d^{P_c}_{n_0}$, which means that the socially optimal lead-time is generally lower than the profit maximizing value for any $n_0$, similarly to the dynamic quotation case.

Furthermore, for $d^{S_c}_{n_0}$ and $d^{S_c}_{n_0+1}$ the above inequality implies that: 
\begin{equation*}
\tilde{d}^{P}_{n_0+1}<d^{S_c}_{n_0+1} \leq \tilde{d}^{P}_{n_0}<d^{S_c}_{n_0} \leq \tilde{d}^{P}_{n_0-1},
\end{equation*}
therefore the optimal single quote is decreasing in $n_0$ and $d^{S_c}_{n_0} \leq d^{P_c}_{n_0}$.

If $n_0=\overline{n}$, then  $B_{\overline{n}}(d)<0$ for all $d \geq 0$.  In this case the feasible interval is $0 \leq d^{S_c}_{\overline{n}} \leq \tilde{d}^{P}_{\overline{n}-1}$, since the $\overline{n}^{\text{th}}$ customer will never join for any $d\geq0$.
If $n_0=\underline{n}$, since $\tilde{d}^{P}_{\underline{n}-1}=\infty$ the feasible interval becomes $\tilde{d}^{P}_{\underline{n}}<d^{S_c}_{\underline{n}} \leq \infty$, where we recall that $d=\infty$ denotes no compensation.

We can summarize the above discussion in the following lemma:
\begin{lemma}\label{optsocc}
\begin{itemize}
\item[(i.)]
For any $n_0 \in \left\lbrace \underline{n}+1, \ldots ,\overline{n}-1 \right\rbrace$, the optimal single lead-time quotation strategy $d^{S_c}_{n_0}$  in \eqref{firstsocproc} subject to the condition $n_0(d)=n_0$ is restricted in the following interval:
\begin{equation*}
d^{S_c}_{n_0}\in (\tilde{d}^{P}_{n_0}, \tilde{d}^{P}_{n_0-1}]= (d^{P_c}_{n_0+1}, d^{P_c}_{n_0}].
\end{equation*}

If $n_0=\underline{n}$ then 
\begin{equation*}
d^{S_c}_{n_0}\in (\tilde{d}^{P}_{n_0}, \infty]=(d^{P_c}_{n_0+1},\infty].
\end{equation*}

If $n_0=\overline{n}$ then 
\begin{equation*}
d^{S_c}_{n_0}\in [0, \tilde{d}^{P}_{n_0-1}]=[0,d^{P_c}_{n_0}].
\end{equation*}

\item[(ii.)] The optimal single lead-time quotation strategy $d^{S_c}_{n_0}$ is decreasing in $n_0$.
\end{itemize}
\end{lemma}

We now consider the computation of the optimal quote.

Let:
\begin{equation}\label{dtildesocc}
a_c(d)=\frac{\overline{F}_{\mu}(d)}{\overline{F}_{v}(d)}e^{-rld}-  \beta e^{-r(R-p)},
\end{equation}
where $v=\mu-r(c-l)$, $\beta=\left(  \frac{\mu-\lambda}{1-\rho^{n_0+1}} \right) \left( \frac{1- \left( \frac{\lambda}{v} \right)^{n_0+1}  }{v-\lambda} \right)$, $\overline{F}_{\mu}(d)=P(X>d)$ and $X$ is the sojourn time of a customer, in a finite capacity $M/M/1/n_0$ queue with service rate $\mu$ in steady state.

\begin{proposition}\label{unimodalc}
For any $n_0 \in \left\lbrace \underline{n}, \ldots ,\overline{n} \right\rbrace$ there are the following cases:
\begin{itemize}
\item[(i.)] If $a_c(\tilde{d}^{P}_{n_0})\leq 0$ then either ${d}^{S_c}_{\overline{n}}=0$, or there is no optimal single lead-time quote such that $n_0(d)=n_0$, i.e., the sup in \eqref{firstsocproc} is not attained by any $d\in(\tilde{d}^{P}_{n_0},\tilde{d}^{P}_{n_0-1}]$. In the second case there exist $\epsilon$-optimal lead-time quotes.
\item[(ii.)] If $a_c(\tilde{d}^{P}_{n_0-1})\geq 0$ then ${d}^{S_c}_{n_0}=\tilde{d}^{P}_{n_0-1}<\infty$ is the unique optimal lead-time quote.
\item[(iii.)] Otherwise ${d}^{S_c}_{n_0}<\infty$ is the unique solution of  $a_c(d)=0$ in  $(\tilde{d}^{P}_{n_0},\tilde{d}^{P}_{n_0-1}]$.
\end{itemize}
\end{proposition}

Proposition \ref{unimodalc} shows that there are values of $n_0$ under which there is no optimal single quote. However in the next Lemma we show that such balking thresholds cannot be optimal at the first stage optimization problem \eqref{social_maxc}. Therefore once the optimal balking threshold is determined, there is always an optimal lead-time quote that can enforce it.

\begin{lemma}\label{noepsilon}
Any $n_0 \in \left\lbrace \underline{n}, \ldots ,\overline{n} \right\rbrace$, that results to $\epsilon$-optimal lead-time quotes, cannot be an optimal balking threshold.
\end{lemma}

For the first-stage problem from Proposition \ref{unimodalc} we have:
\begin{eqnarray*}
\displaystyle Z_{c}(n_0)=S_{c}(n_0,d^{S_c}_{n_0})= \lambda \sum_{n=0}^{n_0-1} q(n;n_0) (G_n+B_n)(d^{S_c}_{n_0})
\end{eqnarray*}

Following the same approach as in the provider's problem in Subsection \ref{singleprovider} we can find the optimal balking threshold $n_{S_c}$ that maximizes $Z_c(n_0)$ through an exhaustive search in the finite interval of balking thresholds which do not results in $\epsilon$-optimal quotes.

Note that there exists at least one such value of $n_0$, namely $n_0=\overline{n}$, because from Lemma \ref{optsocc} the range of feasible quotes in this case is the closed interval $[0, d^{P_c}_{\overline{n}}]$, thus an optimal value exists.

We can now summarize the results in the following theorem:
\begin{theorem}
The problem of social benefit maximization in the class of single lead-time quotes policies has an optimal solution $(d^{S_c}_{n_S},n_{S_c})$ and the maximum social benefit is equal to:
\begin{equation*}
S_{c}^{*}=\lambda \sum_{n=0}^{n_{S_c}-1} q(n;n_{S_c}) (G_n(d^{S_c}_{n_{S_c}})+B_n(d^{S_c}_{n_{S_c}})).
\end{equation*}
\end{theorem}

\section{Effects of Risk Aversion}\label{effectsofriskaversion}
In this section we derive some results regarding the effect of risk aversion on the profit maximization and social optimization problems. Specifically, we explore the effect on: i)balking thresholds and ii)optimal lead-time quotation and profits.

It can be easily shown that both  the lower and upper bound of the balking thresholds are decreasing in $r$.  Fow values of $r$ close to zero we derive the well-known optimal thresholds for the risk-neutral case of \cite{naor1969} without and with compensation respectively for a fixed entrance fee $p$. Specifically:
$\underline{n}(r)<\underline{n}(0)=\floor*{\frac{\mu (R-p)}{c}}$ and $\overline{n}(r)<\overline{n}(0)=\floor*{\frac{\mu (R-p)}{(c-l)}}.$
We can also see that the difference $\overline{n}(r)-\underline{n}(r)$ is increasing in $l$ for any $r$, which means that one effect of the compensation is to increase the range of feasible balking thresholds.
Of course when $l=0$, we have $\underline{n}(r)=\overline{n}(r)$.

Regarding the effect on the optimal lead-time and profits, $\tilde{d}^{P}_{n}(r)$ is decreasing in $r$ since $B_n(d)$ is also decreasing in $r$. The behavior of $\tilde{d}^{S}_{n}(r)$ is not clear analytically, although the numerical results in the next section imply that it is also decreasing. The optimal profits $P^*(r), P^*_{c}(r),S^*(r), S^*_{c}(r)$ are also decreasing in $r$ since the corresponding objective functions and the optimal lead-times are decreasing in $r$.

We finish with a note about the profit and social maximization problems for the risk-neutral case.
The objective function in the profit maximization problem for the provider remains the same as in \eqref{provider_max}:

\begin{equation*}
P(d_n)=\lambda\left( \frac{1-\rho}{1-\rho^{n_0+1}} \right)  \sum_{n=0}^{n_0-1} \rho^n (p-lL_n(d_n)),
\end{equation*}
whereas the social optimization problem  in \eqref{social_max} it is simplified to:

\begin{equation*}
S(d_n)=\lambda\left( \frac{1-\rho}{1-\rho^{n_0+1}} \right) \sum_{n=0}^{n_0-1} \rho^n \left( R-c\frac{n_0+1}{\mu} \right), 
\end{equation*}
which means that the dependence on the quotation policy is only through $n_0$, i.e., the effect on the customer balking threshold. In both cases the necessary condition for joining is $R-p-c\frac{n+1}{\mu}+lL_n(d_n)\geq 0$. 

\section{Computational Experiments}\label{Numericalsection}
In this section we obtain further insights on the four problems analyzed above through computational experiments. First, we observe the sensitivity of the optimal policy and profits with respect to other key parameters of the problem, i.e.,  the entrance fee and the compensation rate. We then perform several experiments to quantify the effect of risk aversion on the optimal policies. Finally, we investigate how the degree of risk aversion affects the minimum required capacity that will allow customers to join the system. We make comparisons between dynamic and single lead-time quotation strategies and investigate the degree of flexibility that the compensation offers.

The base case of parameter values in the following computational experiments is $R=15, c=8, \lambda=10, \mu=12$,  $p=10$, $l=3$, $r=0.5$.

In the first set of experiments we explore the sensitivity of each of the four problems with respect to the entrance fee and the compensation rate. In each case we let one of the parameters vary. In particular, we examine the effect on the optimal provider and social optimizer profits, the range of balking thresholds as well as the optimal values, and the difference between the dynamic and the single lead-time quotation policies. 

Table \ref{ptable} refers to the case where the entrance fee varies from $5$ to $14$. 
First, as $p$ increases, both the lower and upper bound of the balking threshold decrease, and the same holds for the optimal one. When no compensation is offered (Table \ref{withoutcomptable}), we find that the optimal entrance fees are $p=12$ and $p=11$ for the provider and the social optimizer respectively. On the other hand, the option of the compensation allows both the provider and the social optimizer to increase the price to $p=13$ and $p=12$, and at the same time to increase the maximum size of the queue. This results in a significant increase in profits for both the provider and social optimizer. The provider offers higher prices than the social optimizer. Therefore, the option of a compensation lets both the provider and the social optimizer increase the entrance fee and the profits. 

Concerning the optimal thresholds, we observe that for a fixed entrance fee $p$ the provider sets the same or higher values than the social optimizer, which is in the opposite direction with the result of \citep{naor1969}, where the socially optimal threshold is higher. However, when in addition to the lead-time, the entrance fee is also controllable, the thresholds are ordered in the same way as in Naor's model. Besides, the presence of the compensation results in higher queues for the provider and social optimizer. Comparing the two lead-time policies, the extra profit achieved by the dynamic quotation is generally very low.

\begin{table}[]
\centering
\begin{tabular}{|c|c|c|c|c|c|c|}
\hline                                                                                                                                                                              & \multicolumn{3}{c|}{$r=0.5$}                                                                                                                                                      \\   \hline
             
\multirow{-1}{*}{\textbf{$p$}}           
& $P^*$      & $n_0$                     
& $S^*$      \\ \hline

5     & 48.96 &12 & 66.54  \\ \hline

6       & 58.48 &11 & 75.31  \\ \hline

7       & 67.30 & 9 & 83.71  \\ \hline

8       & 76.15 & 8 & 91.47  \\ \hline

9  & 84.54 & 7 & 98.44   \\ \hline

10  & 92.25 & 6 & 104.29 \\ \hline

11  & 95.21 &\cellcolor{cyan}4 & \cellcolor{cyan}106.00  \\ \hline

12  & \cellcolor{yellow}97.64 & \cellcolor{yellow}3 & 105.70   \\ \hline

13  & 94.28 & 2 & 98.96   \\ \hline

14  & 76.36 & 1 & 77.34   \\ \hline
\end{tabular}
  \caption{Optimal thresholds and profits without compensation as a function of $p$.}
  \label{withoutcomptable}
\end{table}

\begin{table}[]
\centering
\begin{tabular}{|c|c|c|c|c|c|c|c|c|c|}
\hline
                       &                               & \multicolumn{4}{c|}{Profit Maximization}                                                                                                                                                      & \multicolumn{4}{c|}{Social Optimization}                                                                                                                                                      \\  \hhline{|>{\arrayrulecolor{white}}->{\arrayrulecolor{black}}|>{\arrayrulecolor{white}}->{\arrayrulecolor{black}}|--------}
                       &                               & \multicolumn{2}{c|}{\cellcolor[HTML]{FFFFC7}Dynamic}                                          & \multicolumn{2}{c|}{\cellcolor[HTML]{ECF4FF}Single}                                         & \multicolumn{2}{c|}{\cellcolor[HTML]{FFFFC7}Dynamic}                                          & \multicolumn{2}{c|}{\cellcolor[HTML]{ECF4FF}Single}                                         \\  \hhline{|>{\arrayrulecolor{white}}->{\arrayrulecolor{black}}|>{\arrayrulecolor{white}}->{\arrayrulecolor{black}}|--------} 
                       
\multirow{-3}{*}{$p$}    & \multirow{-3}{*}{$[\underline{n},\overline{n}]$} & \cellcolor[HTML]{FFFFC7}$n_P$                     & \cellcolor[HTML]{FFFFC7}$P^*$                     & \cellcolor[HTML]{ECF4FF}$n_{P_c}$                     & \cellcolor[HTML]{ECF4FF}$P^{*}_{c}$                     & \cellcolor[HTML]{FFFFC7}$n_S$                     & \cellcolor[HTML]{FFFFC7}$S^*$                     & \cellcolor[HTML]{ECF4FF}$n_{S_c}$                     & \cellcolor[HTML]{ECF4FF}$S^{*}_{c}$                     \\ \hline

5 & [12,21]         & {\cellcolor[HTML]{FFFFC7}15} & {\cellcolor[HTML]{FFFFC7}49.24} & {\cellcolor[HTML]{ECF4FF}13} & {\cellcolor[HTML]{ECF4FF}49.12} & {\cellcolor[HTML]{FFFFC7}12} & {\cellcolor[HTML]{FFFFC7}66.79} & {\cellcolor[HTML]{ECF4FF}13} & {\cellcolor[HTML]{ECF4FF}66.71} \\ \hline

6 & [11,19]         & {\cellcolor[HTML]{FFFFC7}14} & {\cellcolor[HTML]{FFFFC7}58.86} & {\cellcolor[HTML]{ECF4FF}12} & {\cellcolor[HTML]{ECF4FF}58.68} & {\cellcolor[HTML]{FFFFC7}11} & {\cellcolor[HTML]{FFFFC7}75.64} & {\cellcolor[HTML]{ECF4FF}12} & {\cellcolor[HTML]{ECF4FF}75.58} \\ \hline

7 &[9,17]         & {\cellcolor[HTML]{FFFFC7}12} & {\cellcolor[HTML]{FFFFC7}68.32} & {\cellcolor[HTML]{ECF4FF}11} & {\cellcolor[HTML]{ECF4FF}68.04} & {\cellcolor[HTML]{FFFFC7}10} & {\cellcolor[HTML]{FFFFC7}84.10} & {\cellcolor[HTML]{ECF4FF}11} & {\cellcolor[HTML]{ECF4FF}84.07} \\ \hline

8 & [8,14]         & {\cellcolor[HTML]{FFFFC7}11} & {\cellcolor[HTML]{FFFFC7}77.55} & {\cellcolor[HTML]{ECF4FF}10} & {\cellcolor[HTML]{ECF4FF}77.11} & {\cellcolor[HTML]{FFFFC7}10} & {\cellcolor[HTML]{FFFFC7}92.07} & {\cellcolor[HTML]{ECF4FF}10} & {\cellcolor[HTML]{ECF4FF}92.05} \\ \hline

9 & [7,12]         & {\cellcolor[HTML]{FFFFC7}10} & {\cellcolor[HTML]{FFFFC7}86.47} & {\cellcolor[HTML]{ECF4FF}9} & {\cellcolor[HTML]{ECF4FF}85.73} & {\cellcolor[HTML]{FFFFC7}9} & {\cellcolor[HTML]{FFFFC7}99.40} & {\cellcolor[HTML]{ECF4FF}9} & {\cellcolor[HTML]{ECF4FF}99.38} \\ \hline

10 & [6,10]     & {\cellcolor[HTML]{FFFFC7}9} & {\cellcolor[HTML]{FFFFC7}94.91} & {\cellcolor[HTML]{ECF4FF}8} & {\cellcolor[HTML]{ECF4FF}93.66} & {\cellcolor[HTML]{FFFFC7}8} & {\cellcolor[HTML]{FFFFC7}105.86} & {\cellcolor[HTML]{ECF4FF}8} & {\cellcolor[HTML]{ECF4FF}105.80} \\ \hline

11 & [4,8]         & {\cellcolor[HTML]{FFFFC7}8} & {\cellcolor[HTML]{FFFFC7}102.68} & {\cellcolor[HTML]{ECF4FF}6} & {\cellcolor[HTML]{ECF4FF}100.64} & {\cellcolor[HTML]{FFFFC7}7} & {\cellcolor[HTML]{FFFFC7}111.15} & {\cellcolor[HTML]{ECF4FF}7} & {\cellcolor[HTML]{ECF4FF}110.91} \\ \hline

12 & [3,6]         & {\cellcolor[HTML]{FFFFC7}6} & {\cellcolor[HTML]{FFFFC7}108.74} & {\cellcolor[HTML]{ECF4FF}5} & {\cellcolor[HTML]{ECF4FF}106.06} & {\cellcolor[HTML]{FFFFC7}6} & \cellcolor{yellow}114.90 & {\cellcolor[HTML]{ECF4FF}6} & \cellcolor{cyan}114.12 \\ \hline

13 & [2,4]         & {\cellcolor[HTML]{FFFFC7}4} &\cellcolor{yellow}110.73 & {\cellcolor[HTML]{ECF4FF}4} & \cellcolor{cyan}108.64 & {\cellcolor[HTML]{FFFFC7}4} & {\cellcolor[HTML]{FFFFC7}114.31} & {\cellcolor[HTML]{ECF4FF}4} & {\cellcolor[HTML]{ECF4FF}114.03} \\ \hline

14 & [1,2]         & {\cellcolor[HTML]{FFFFC7}2} & {\cellcolor[HTML]{FFFFC7}100.10} & {\cellcolor[HTML]{ECF4FF}2} & {\cellcolor[HTML]{ECF4FF}99.33} & {\cellcolor[HTML]{FFFFC7}2} & {\cellcolor[HTML]{FFFFC7}101.04} & {\cellcolor[HTML]{ECF4FF}2} & {\cellcolor[HTML]{ECF4FF}101.01} \\ \hline
\end{tabular}
  \caption{Optimal thresholds and profits as a function of $p$.}
  \label{ptable}
\end{table}

In Figure \ref{pvariable}, we present the sequence of the optimal dynamic lead-time quotes and the optimal single quote. The left graph corresponds to the provider maximization problem and the right graph to the social optimization problem. The asterisks denote the optimal dynamic lead-time quotes, the crosses are the optimal single quotes, and the circle denotes that at the balking threshold the quote must be higher than the plotted value to discourage the incoming customers. Note that for the provider maximization problem, the optimal dynamic lead-time quotes for states $n<\underline{n}$ are $d=\infty$, while for the social optimization problem they are finite.
In general, the service provider offers the highest possible lead-time quotes to make customers indifferent, while the social optimizer keeps a balance between the provider's objective for higher profits and the customers' cost of delay. Furthermore, when the entrance fee is high, both the provider and social optimizer offer lower optimal dynamic and single lead-time quotes to balance the delay aversion of potential customers and entice them to enter the queue.

\begin{figure}
\includegraphics[width=7cm,height=7cm]{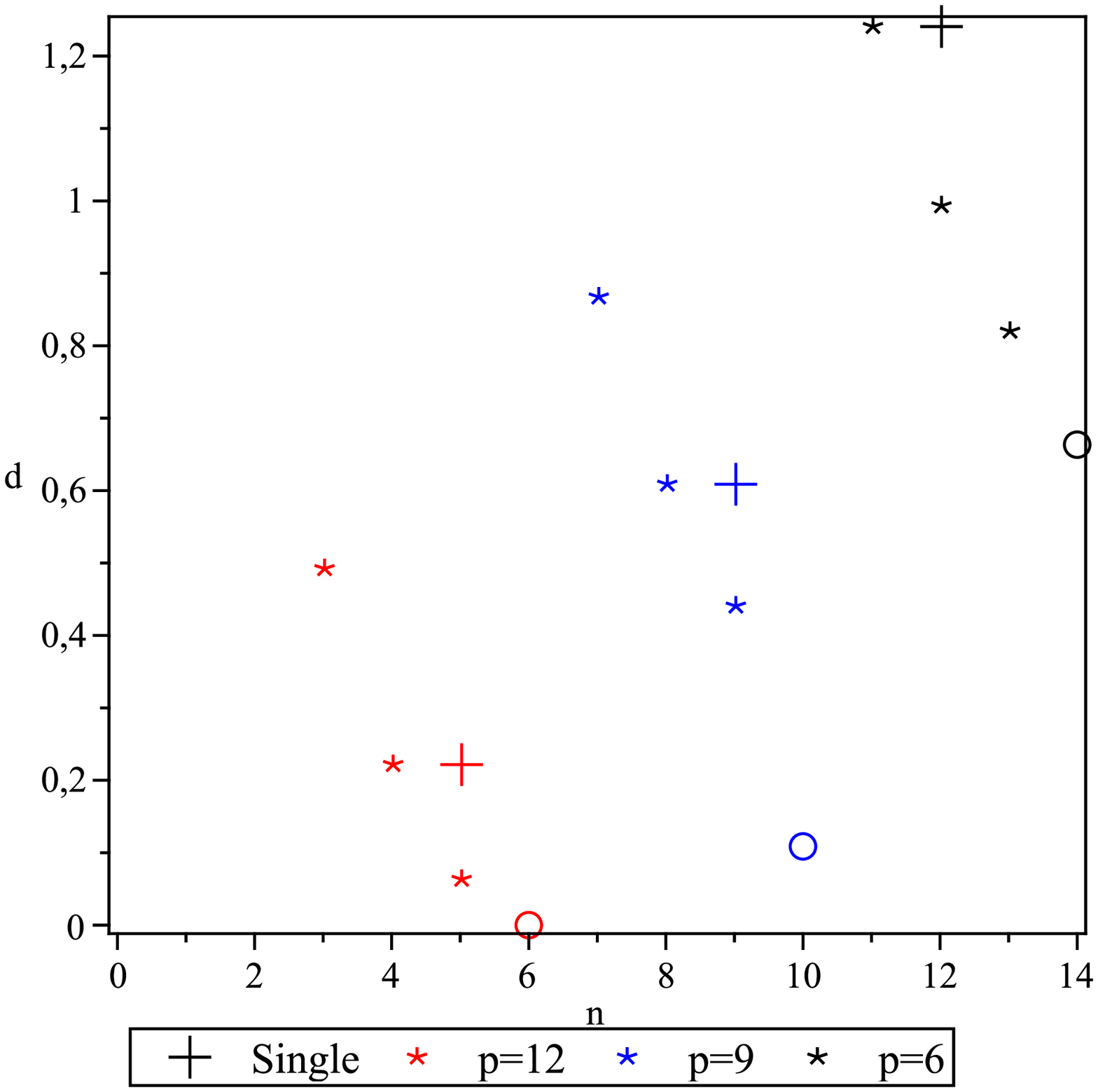}
\includegraphics[width=7cm,height=7cm]{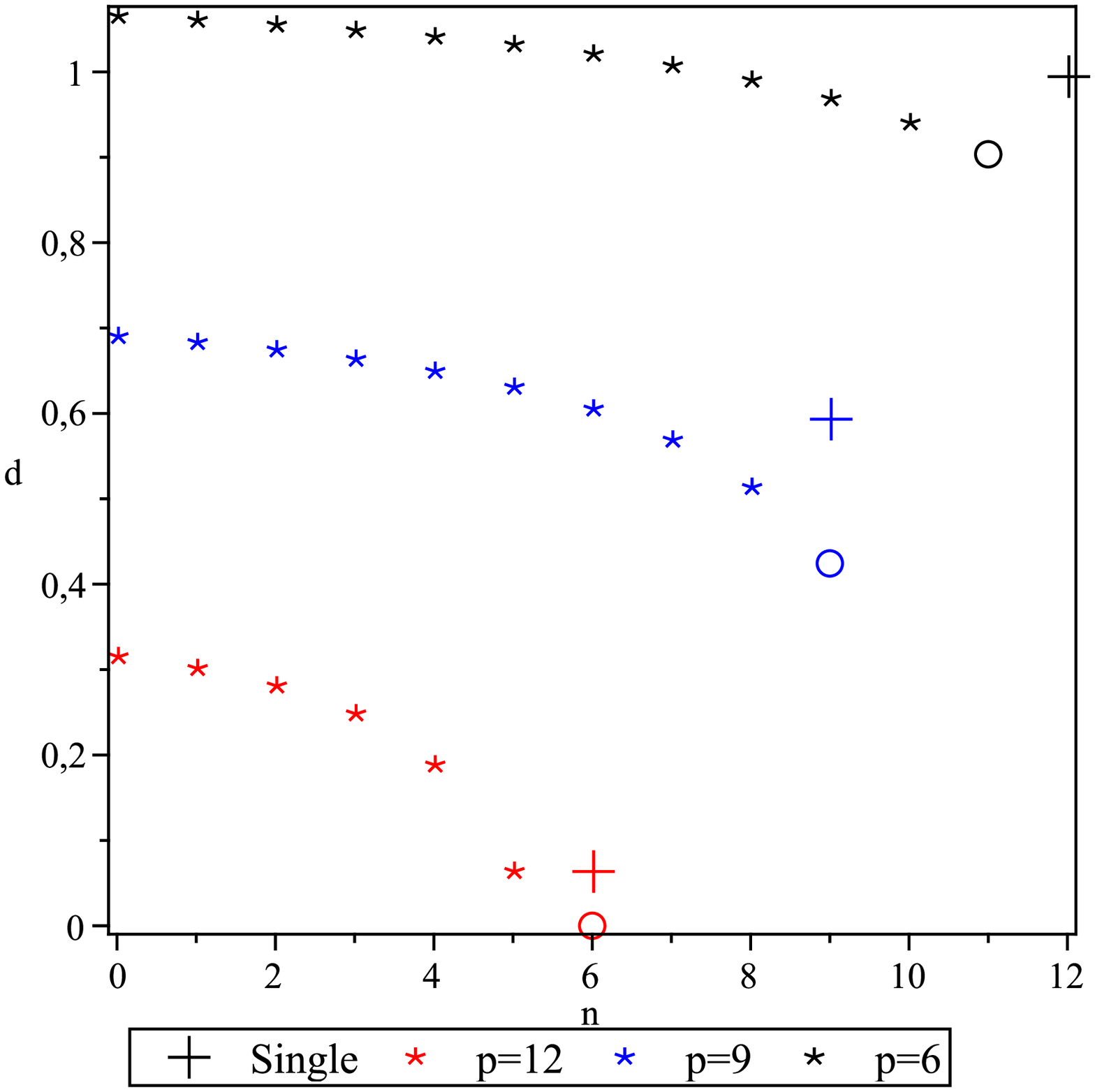}
\caption{Optimal single and dynamic lead-time quotes for profit maximization (left) and social optimization (right).}
\label{pvariable}
\end{figure}

Table \ref{ltable} refers to the case when the compensation rate varies from $0$ to full compensation. We first observe that as the compensation increases, the upper bound of the balking threshold is also increasing. The same holds for the optimal threshold. One interesting insight from this table is that both profits are increasing in $l$, which means that it is optimal both for the provider and the social optimizer to offer full compensation. On the other hand, the profit variation is very low in the entire range of $l$, from no to full compensation. Compared to the case of variable entrance fee, we can see that the main role of the compensation is to let the provider and the social optimizer maximize profits by offering sufficient lead-time quotes and higher entrance fees. Finally, as in the previous experiment, the optimal maximum size of the queue is usually longer under the provider's problem and that the choice of dynamic or single quotation policy does not really affect the optimal profits.

Figure \ref{lvariable} goes along similar lines with Figure \ref{pvariable}. The only difference is that the optimal lead-time quotes are increasing in $l$. This is expected, since a higher compensation allows the provider and the social optimizer to increase the lead-times without deterring customers.

\begin{table}[]
\centering
\begin{tabular}{|c|c|c|c|c|c|c|c|c|c|}
\hline
                       &                               & \multicolumn{4}{c|}{Profit Maximization}                                                                                                                                                      & \multicolumn{4}{c|}{Social Optimization}                                                                                                                                                      \\  \hhline{|>{\arrayrulecolor{white}}->{\arrayrulecolor{black}}|>{\arrayrulecolor{white}}->{\arrayrulecolor{black}}|--------}
                       &                               & \multicolumn{2}{c|}{\cellcolor[HTML]{FFFFC7}Dynamic}                                          & \multicolumn{2}{c|}{\cellcolor[HTML]{ECF4FF}Single}                                         & \multicolumn{2}{c|}{\cellcolor[HTML]{FFFFC7}Dynamic}                                          & \multicolumn{2}{c|}{\cellcolor[HTML]{ECF4FF}Single}                                         \\  \hhline{|>{\arrayrulecolor{white}}->{\arrayrulecolor{black}}|>{\arrayrulecolor{white}}->{\arrayrulecolor{black}}|--------} 
                       
\multirow{-3}{*}{$l$}    & \multirow{-3}{*}{$[\underline{n},\overline{n}]$} & \cellcolor[HTML]{FFFFC7}$n_P$                     & \cellcolor[HTML]{FFFFC7}$P^*$                     & \cellcolor[HTML]{ECF4FF}$n_{P_c}$                     & \cellcolor[HTML]{ECF4FF}$P^{*}_{c}$                     & \cellcolor[HTML]{FFFFC7}$n_S$                     & \cellcolor[HTML]{FFFFC7}$S^*$                     & \cellcolor[HTML]{ECF4FF}$n_{S_c}$                     & \cellcolor[HTML]{ECF4FF}$S^{*}_{c}$                     \\ \hline

0 & [6,6]         & {\cellcolor[HTML]{FFFFC7}6} & {\cellcolor[HTML]{FFFFC7}92.25} & {\cellcolor[HTML]{ECF4FF}6} & {\cellcolor[HTML]{ECF4FF}92.25} & {\cellcolor[HTML]{FFFFC7}6} & {\cellcolor[HTML]{FFFFC7}104.29} & {\cellcolor[HTML]{ECF4FF}6} & {\cellcolor[HTML]{ECF4FF}104.29} \\ \hline

2 & [6,8]         & {\cellcolor[HTML]{FFFFC7}8} & {\cellcolor[HTML]{FFFFC7}94.47} & {\cellcolor[HTML]{ECF4FF}7} & {\cellcolor[HTML]{ECF4FF}93.44} & {\cellcolor[HTML]{FFFFC7}8} & {\cellcolor[HTML]{FFFFC7}105.60} & {\cellcolor[HTML]{ECF4FF}7} & {\cellcolor[HTML]{ECF4FF}105.42} \\ \hline

4 & [6,13]         & {\cellcolor[HTML]{FFFFC7}10} & {\cellcolor[HTML]{FFFFC7}95.09} & {\cellcolor[HTML]{ECF4FF}8} & {\cellcolor[HTML]{ECF4FF}94.04} & {\cellcolor[HTML]{FFFFC7}8} & {\cellcolor[HTML]{FFFFC7}106.01} & {\cellcolor[HTML]{ECF4FF}8} & {\cellcolor[HTML]{ECF4FF}105.99} \\ \hline

6 & [6,28]         & {\cellcolor[HTML]{FFFFC7}10} & {\cellcolor[HTML]{FFFFC7}95.28} & {\cellcolor[HTML]{ECF4FF}9} & {\cellcolor[HTML]{ECF4FF}94.38} & {\cellcolor[HTML]{FFFFC7}9} & {\cellcolor[HTML]{FFFFC7}106.16} & {\cellcolor[HTML]{ECF4FF}9} & {\cellcolor[HTML]{ECF4FF}106.16} \\ \hline

8 & [6,$\infty$]         & {\cellcolor[HTML]{FFFFC7}10} & \cellcolor{yellow}95.32 & {\cellcolor[HTML]{ECF4FF}10} & \cellcolor{cyan}94.58 & {\cellcolor[HTML]{FFFFC7}9} & \cellcolor{yellow}106.21 & {\cellcolor[HTML]{ECF4FF}9} & \cellcolor{cyan}106.20 \\ \hline
\end{tabular}
  \caption{Optimal thresholds and profits as a function of $l$.}
  \label{ltable}
\end{table}

\begin{figure}
\includegraphics[width=7cm,height=7cm]{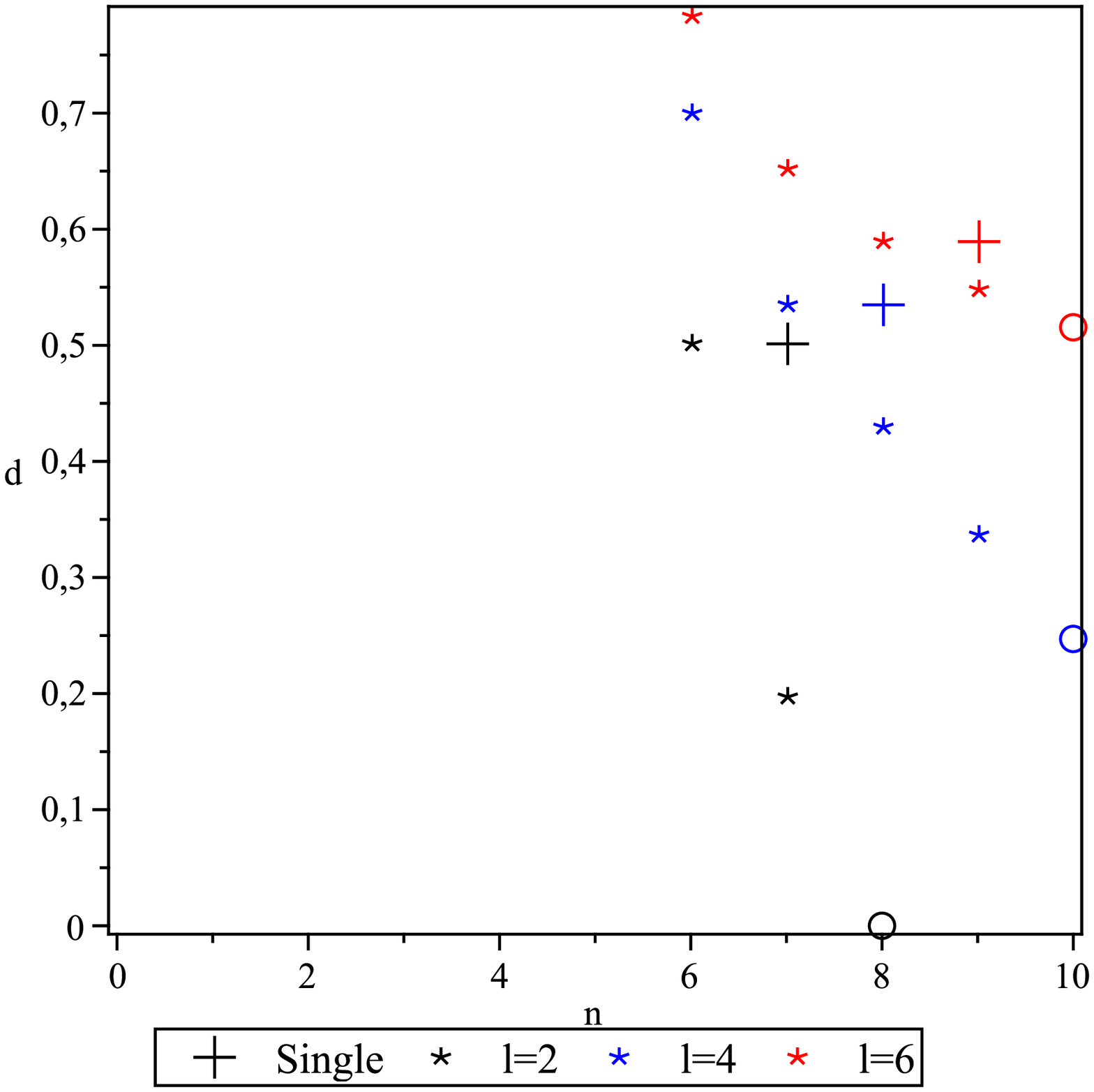}
\includegraphics[width=7cm,height=7cm]{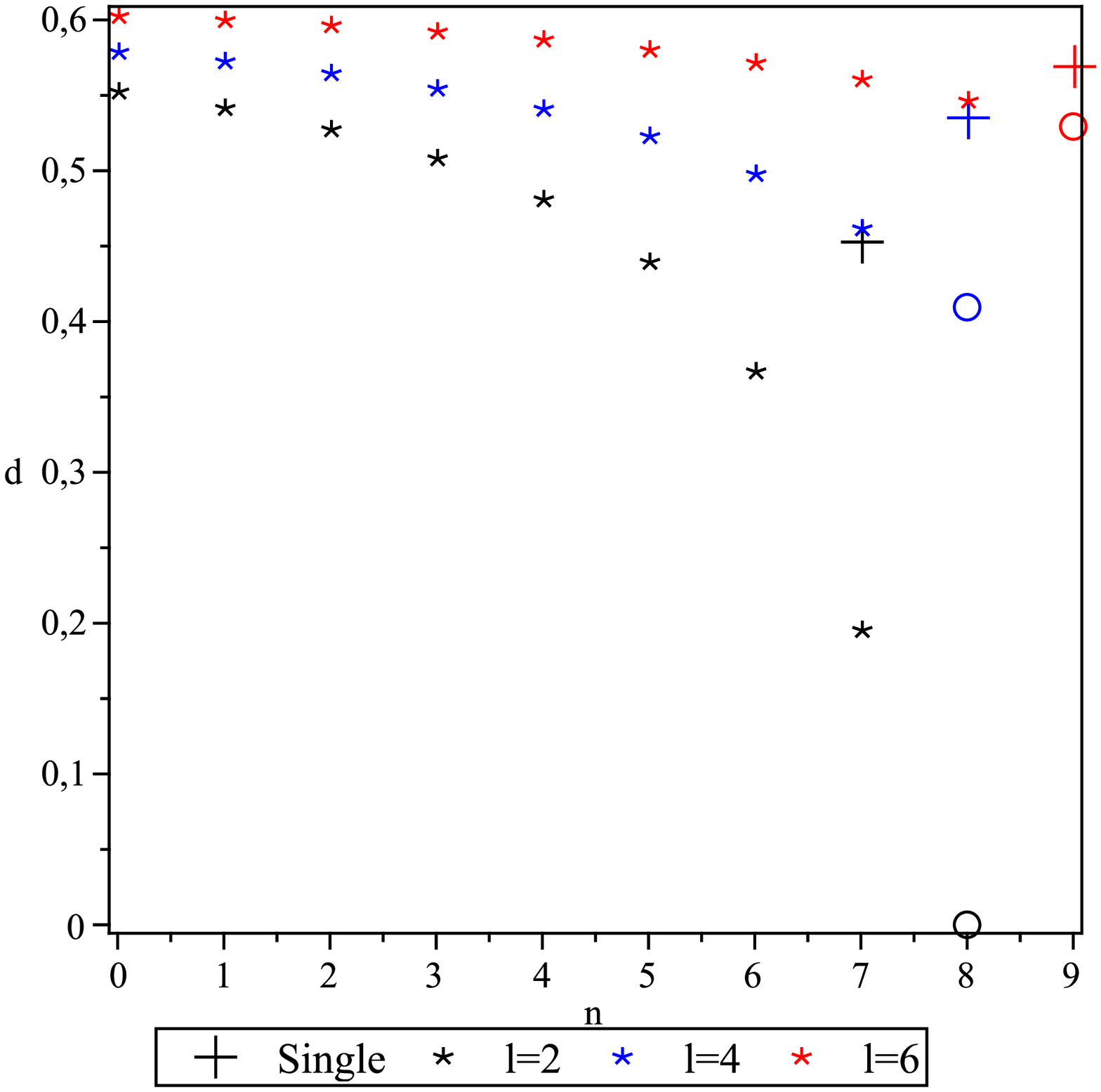}
\caption{Optimal single and dynamic lead-time quotes for profit maximization (left) and social optimization (right).}
\label{lvariable}
\end{figure}
%end of part1
In the second set of experiments, we investigate numerically the effect of risk aversion on the optimal policy and profits.
In Table \ref{rtable} we present the optimal dynamic and single lead-time quotes for $r=0$ and a high risk aversion coefficient $r=1.3$. The optimal balking threshold for the provider is reduced from $10$ to $8$ when $r=1.3$ under the dynamic policy (from $8$ to $6$ under the single policy), compared to the risk-neutral $r=0$. On the other hand, the social optimizer increases the threshold by one under the dynamic policy, and keeps it the same under single quote. This can be explained by the different incentives between the provider and the social optimizer. The latter offers lower quotes and generally leaves a positive benefit to the customers. From several experiments, we observe that the optimal threshold is decreasing in $r$ for the profit maximization problem. However, this does not hold necessarily under the social optimization. As in the cases of variable $p$ and $l$, the provider may prefer longer queues than the social optimizer, and the dynamic policy also induces longer queues compared to the single quote.

\begin{table}[]
\centering
\begin{tabular}{|c|c|c|c|c|c|c|c|c|}
\hline
                        & \multicolumn{4}{c|}{$r=0$}                                                                                                                                                      & \multicolumn{4}{c|}{$r=1.3$}                                                                                                                                                      \\   \hline

                                                    & \multicolumn{2}{c|}{\cellcolor[HTML]{FFFFC7}Profit Max}                                          & \multicolumn{2}{c|}{\cellcolor[HTML]{ECF4FF}Social Opt}                                         & \multicolumn{2}{c|}{\cellcolor[HTML]{FFFFC7}Profit Max}                                          & \multicolumn{2}{c|}{\cellcolor[HTML]{ECF4FF}Social Opt}                                         \\ \hline
                       
\multirow{-1}{*}{\textbf{$n$}}       & \cellcolor[HTML]{FFFFC7}$\tilde{d}^{P}_{n}$  & \cellcolor[HTML]{FFFFC7}$d^{P_c}_{n_0}$              & \cellcolor[HTML]{ECF4FF}$\tilde{d}^{S}_{n}$                     & \cellcolor[HTML]{ECF4FF}$d^{S_c}_{n_0}$                     & \cellcolor[HTML]{FFFFC7}$\tilde{d}^{P}_{n}$                     & \cellcolor[HTML]{FFFFC7}$d^{P_c}_{n_0}$                     & \cellcolor[HTML]{ECF4FF}$\tilde{d}^{S}_{n}$                     & \cellcolor[HTML]{ECF4FF}$d^{S_c}_{n_0}$ \\ \hline

0         & {\cellcolor[HTML]{FFFFC7}$\infty$} & {\cellcolor[HTML]{FFFFC7}0.62} & {\cellcolor[HTML]{ECF4FF}0.57} & {\cellcolor[HTML]{ECF4FF}0.63} & {\cellcolor[HTML]{FFFFC7}$\infty$} & {\cellcolor[HTML]{FFFFC7}0.46} & {\cellcolor[HTML]{ECF4FF}0.54} & {\cellcolor[HTML]{ECF4FF}0.26} \\ \hline

1         & {\cellcolor[HTML]{FFFFC7}$\infty$} & {\cellcolor[HTML]{FFFFC7}0.62}  & {\cellcolor[HTML]{ECF4FF}0.56} & {\cellcolor[HTML]{ECF4FF}0.63}  & {\cellcolor[HTML]{FFFFC7}$\infty$} & {\cellcolor[HTML]{FFFFC7}0.46} & {\cellcolor[HTML]{ECF4FF}0.53} & {\cellcolor[HTML]{ECF4FF}0.26} \\ \hline

2         & {\cellcolor[HTML]{FFFFC7}$\infty$} & {\cellcolor[HTML]{FFFFC7}0.62} & {\cellcolor[HTML]{ECF4FF}0.55} & {\cellcolor[HTML]{ECF4FF}0.63} & {\cellcolor[HTML]{FFFFC7}$\infty$} & {\cellcolor[HTML]{FFFFC7}0.46} & {\cellcolor[HTML]{ECF4FF}0.51} & {\cellcolor[HTML]{ECF4FF}0.26} \\ \hline

3       & {\cellcolor[HTML]{FFFFC7}$\infty$} & {\cellcolor[HTML]{FFFFC7}0.62} & {\cellcolor[HTML]{ECF4FF}0.54} & {\cellcolor[HTML]{ECF4FF}0.63}  & {\cellcolor[HTML]{FFFFC7}1.20} & {\cellcolor[HTML]{FFFFC7}0.46} & {\cellcolor[HTML]{ECF4FF}0.48} & {\cellcolor[HTML]{ECF4FF}0.26} \\ \hline

4       & {\cellcolor[HTML]{FFFFC7}$\infty$} & {\cellcolor[HTML]{FFFFC7}0.62} & {\cellcolor[HTML]{ECF4FF}0.53} & {\cellcolor[HTML]{ECF4FF}0.63}  & {\cellcolor[HTML]{FFFFC7}0.71} & {\cellcolor[HTML]{FFFFC7}0.46} & {\cellcolor[HTML]{ECF4FF}0.44} & {\cellcolor[HTML]{ECF4FF}0.26} \\ \hline

5       & {\cellcolor[HTML]{FFFFC7}$\infty$} & {\cellcolor[HTML]{FFFFC7}0.62} & {\cellcolor[HTML]{ECF4FF}0.51} & {\cellcolor[HTML]{ECF4FF}0.63}  & {\cellcolor[HTML]{FFFFC7}0.46} & {\cellcolor[HTML]{FFFFC7}0.46} & {\cellcolor[HTML]{ECF4FF}0.37} & {\cellcolor[HTML]{ECF4FF}0.26} \\ \hline

6       & {\cellcolor[HTML]{FFFFC7}$\infty$} & {\cellcolor[HTML]{FFFFC7}0.62} & {\cellcolor[HTML]{ECF4FF}0.48} & {\cellcolor[HTML]{ECF4FF}0.63}  & {\cellcolor[HTML]{FFFFC7}0.26} & {\cellcolor{yellow}-} & {\cellcolor[HTML]{ECF4FF}0.25} & {\cellcolor[HTML]{ECF4FF}0.26} \\ \hline

7       & {\cellcolor[HTML]{FFFFC7}0.62} & {\cellcolor[HTML]{FFFFC7}0.62} & {\cellcolor{cyan}-} & {\cellcolor{cyan}-}  & {\cellcolor[HTML]{FFFFC7}0.06} & {\cellcolor[HTML]{FFFFC7}-} & {\cellcolor[HTML]{ECF4FF}0.06} & {\cellcolor{cyan}-} \\ \hline

8       & {\cellcolor[HTML]{FFFFC7}0.42} & {\cellcolor{yellow}-} & {\cellcolor[HTML]{ECF4FF}-} & {\cellcolor[HTML]{ECF4FF}-}  & {\cellcolor{yellow}-} & {\cellcolor[HTML]{FFFFC7}-} & {\cellcolor{cyan}-} & {\cellcolor[HTML]{ECF4FF}-} \\ \hline

9       & {\cellcolor[HTML]{FFFFC7}0.27} & {\cellcolor[HTML]{FFFFC7}-} & {\cellcolor[HTML]{ECF4FF}-} & {\cellcolor[HTML]{ECF4FF}-}  & {\cellcolor[HTML]{FFFFC7}-} & {\cellcolor[HTML]{FFFFC7}-} & {\cellcolor[HTML]{ECF4FF}-} & {\cellcolor[HTML]{ECF4FF}-} \\ \hline

10       & {\cellcolor{yellow}-} & {\cellcolor[HTML]{FFFFC7}-} & {\cellcolor[HTML]{ECF4FF}-} & {\cellcolor[HTML]{ECF4FF}-}  & {\cellcolor[HTML]{FFFFC7}-} & {\cellcolor[HTML]{FFFFC7}-} & {\cellcolor[HTML]{ECF4FF}-} & {\cellcolor[HTML]{ECF4FF}-} \\ \hline
\end{tabular}
  \caption{Optimal lead-times quotes.}
  \label{rtable}
\end{table}

\begin{figure}
\centering
\includegraphics[width=7cm,height=7cm]{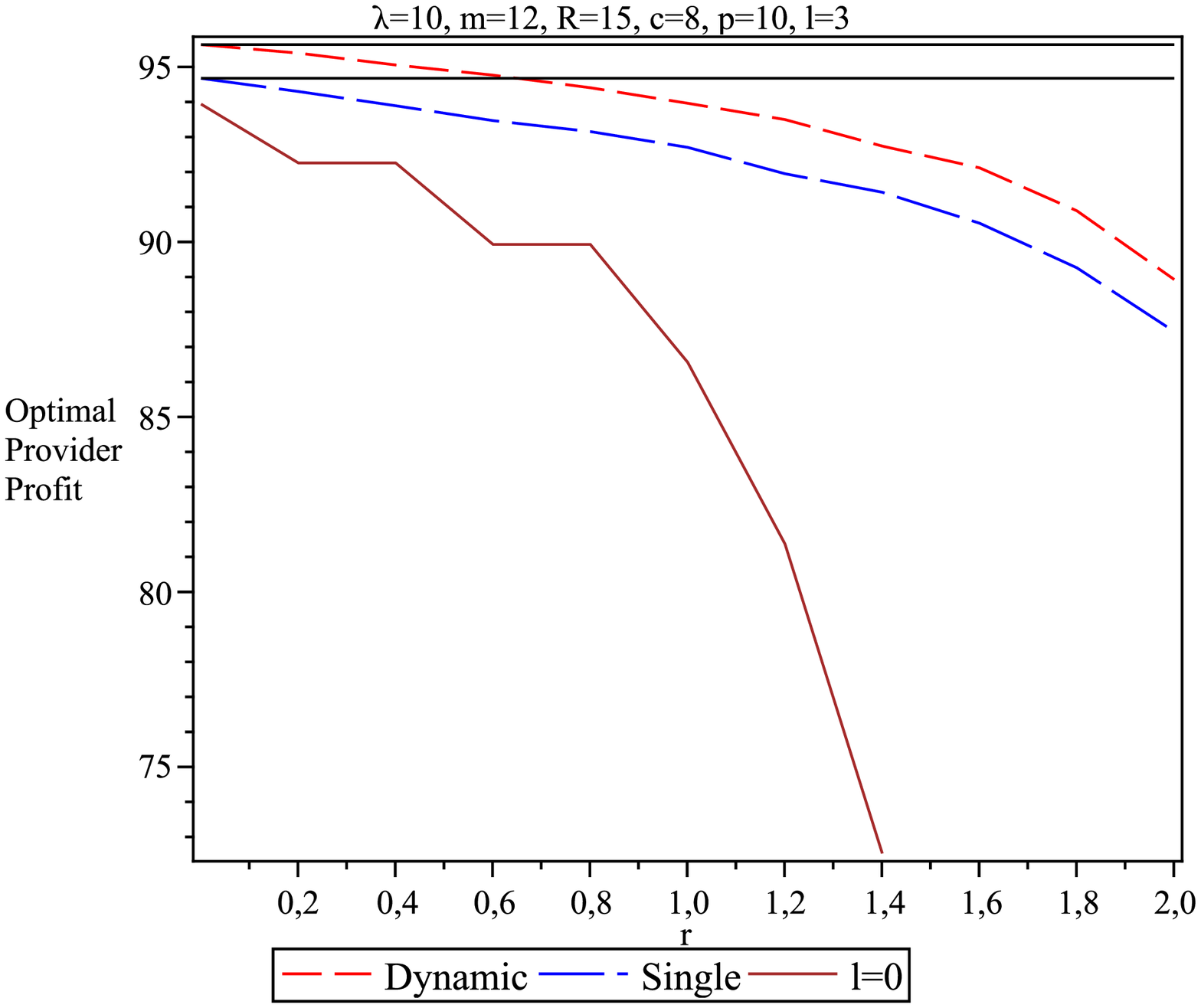}
\includegraphics[width=7cm,height=7cm]{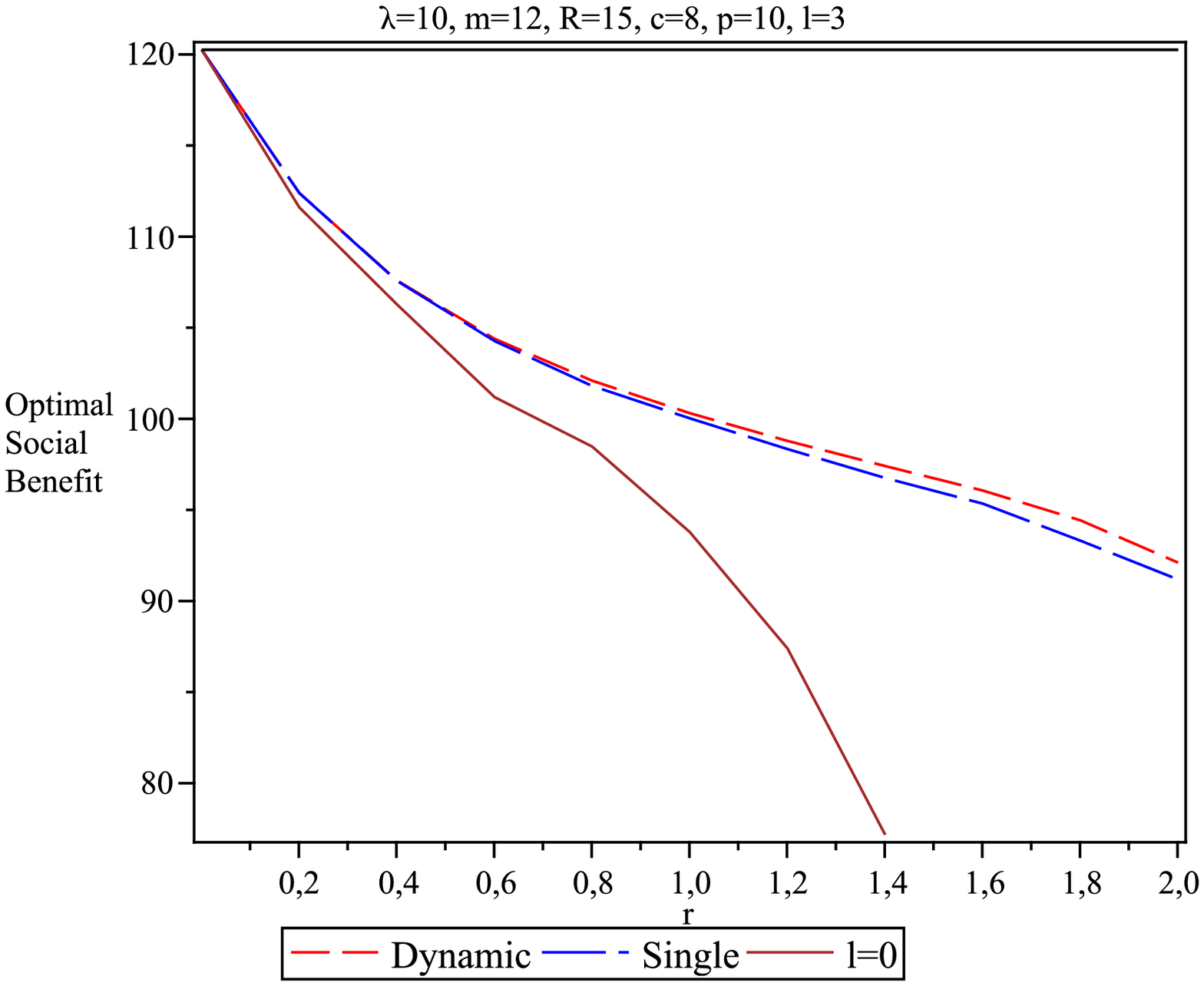}
\caption{The effect of risk aversion on optimal profits.}
\label{ProviderRA}
\end{figure}

In Figure \ref{ProviderRA} we present the optimal profits for the provider and social optimizer as a function of the risk aversion parameter $r$, under the optimal dynamic and single quotes policies as well as for the case of no compensation. These are compared to the benchmark case of no risk-aversion, indicated by the horizontal lines at the corresponding profit levels at $r=0$.

The left graph presents the provider's side. We observe that the profits under the dynamic and single policy are very close. Additionally, the presence of the compensation allows the provider to retain a significantly higher portion of the risk-neutral profit as the risk aversion increases, as well as serve customers under values of r higher than $\frac{\mu}{c}$, i.e., up to $\frac{\mu}{c-l}$. An interesting point here is that for low degrees of risk aversion, the dynamic policy leads to slightly more profits than the risk-neutral case for a single quote. Generally it is clear that both policies, dynamic and single, may substantially alleviate the detrimental effects of risk aversion.

From the social point of view, in the right graph, we observe that the difference in profits between the dynamic and the single policy are even smaller than in the provider's problem. In the presence of compensation the social optimizer can also offer appropriate lead-time quotes to gain more profits and service customers with higher degrees of aversion. Finally, under risk neutrality, profits are identical for both policies. However, the main difference with the provider's side is that the effect of risk aversion is significantly more intense, which means that the provider can handle the effect of aversion of customers more effectively than the social optimizer.

%end of part 2 -min req capacity
In the last set of experiments, we explore the effect of risk aversion on the minimum required capacity that a customer joins, i.e., the required service rate when the system is empty before the first customer arrives. In Figure \ref{mincapacityfig}, we present the minimum capacity that is required under values of $r$, $0$, $0.1$ and $0.5$.
The interesting point here is that the minimum capacity is quite sensitive to variations of the risk aversion coefficient. The fact that the risk-averse customers are more reluctant to join, and the absence of compensation forces the provider and the social optimizer to increase the service rate. On the other hand, when customers are risk-neutral, and at the same time they receive compensation, their expected benefit is sufficiently large, so the minimum required capacity can be significantly low. 

\begin{figure}
\centering
\includegraphics[width=9cm,height=7cm]{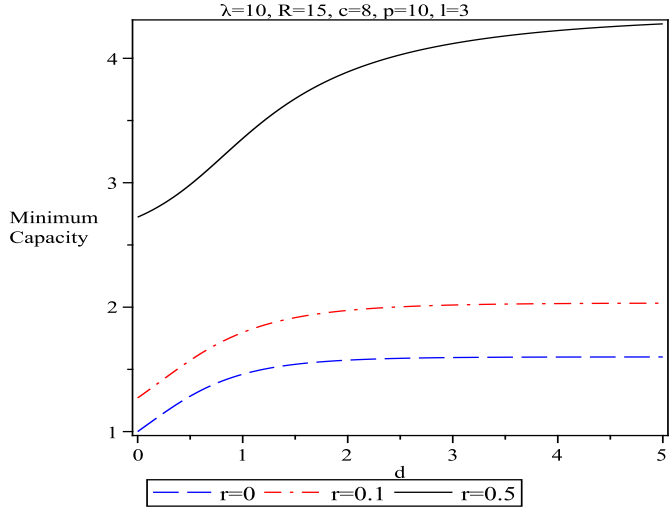}
\caption{Minimum required capacity as a function of lead-time.}
\label{mincapacityfig}
\end{figure}

We summarize the most interesting insights of this section:
\begin{enumerate}
\item The lead-time quotation policy approach may lead to longer queues than is socially optimal.

\item Varying the entrance fee causes significant changes in profits both for the provider and the social optimizer.

\item Varying the compensation rate has a less drastic effect on the profits than varying the entrance fee. The main benefit of the compensation is that it allows the provider and the social optimizer to set higher entrance fees without losing customers, even when they are risk-averse. Additionally, it enables accepting customers under a larger range of risk aversion parameter values.
 
\item By setting an appropriate single lead-time quote, both the provider and the social optimizer can earn almost the same profit as under a dynamic policy. This is important, because applying a dynamic lead-time policy may require more expensive equipment and higher effort, and may give a sense of unequal treatment to incoming customers.

\item The consequences of risk aversion can be addressed more effectively by both the dynamic and single strategy for the profit maximization problem compared to the social benefit optimization.

\item The minimum required capacity to entice customers to join the system is very sensitive to the level of risk aversion.
\end{enumerate}

\section{Conclusions}\label{Conclusions}
In this paper we develop customer equilibrium strategies and profit maximizing policies for an observable Make-To-Order/service production system with risk-averse customers, where the service provider and the social optimizer, adopts a lead-time quotation strategy, i.e., dynamic or single, and a corresponding balking threshold. In addition to the lead-time there is also a fixed entrance fee and a delay compensation for the excess delay above the lead-time quote. The risk aversion is modeled by a CARA utility function of the customer net benefit. 

We analyze and solve four problems and we compare the optimal lead-time/threshold strategies in each case. Under the dynamic quotation policy, we present an algorithm both for the provider and social optimizer where we derived the optimal balking thresholds, while for the single quotation problems, the optimal quotation strategy can be found by exhaustive searching algorithms. We also explore the effect of risk aversion respect to the thresholds, minimum required capacity and optimal profits, compared to the risk-neutral case. 
In the numerical section, we illustrate several tables and plots to quantify the effect of risk aversion on the optimal policies and profits, as well as a sensitivity analysis for the other two key parameters of the system, i.e., the entrance fee and the compensation rate. It is shown that the lead-time quotation policy may lead to longer queues than it is socially desired which is a contradictory result of Naor where there is control of the entrance fee instead. Moreover, the provider can address the risk aversion of customers more efficiently than the social optimizer. Regarding the question about which policy is more profitably the answer is dynamic, but the differences with the single one is negligible. 
Finally, we observe how the maximum capacity of the queue is drastically increasing in the risk aversion coefficient.
 
This work could be extended in several directions. A natural extension is to expand this model under several classes of customers regarding one of their characteristics, e.g, their waiting cost. Furthermore, it would be interesting to consider the case of a Make-To-Stock system where the provider has also to decide about the optimal base stock level and compare the results with the Make-To-Order strategy. Finally, the lead-time quotation/compensation approach could be studied in a queueing network environment, such as an unobservable system of $N$ heterogeneous queues in series analyzed in \cite{Burnetas2013}. In the network setting the optimal quotation problem obtains a dynamic nature due to the sequence of join/balk customer decisions, although the system state itself is unobservable.

\section*{Appendix A}

\begin{proof}[Proof of Lemma \ref{monotonicity}] 
\begin{itemize}
\item[$(i.)$]
From known properties of the Gamma distribution it follows $X_n \leq_{st} X_{n+1}.$ In addition $U(X)$ is decreasing in $X$ since $l \leq c$. Thus, $E\left(U(X_{n+1})\right) \leq E\left(U(X_{n})\right)$ and $B_{n+1}(d) \leq B_n(d)$, for all $d.$
The monotonicity of $B_n(d)$ in terms of $d$ for a fixed $n$ is immediate.
\item[$(ii.)$]
For a fixed $d$, since the function $(X-d)^{+} \text{ is increasing in $X$ and}$ $X_n \leq_{st} X_{n+1},$ it follows that $E(X_n-d)^{+} \leq E(X_{n+1}-d)^{+}.$ Therefore $L_n(d) \leq L_{n+1}(d),$ for all $d.$
The monotonicity of $L_n(d)$ in terms of $d$ for a fixed $n$ is immediate.
\item[$(iii.)$]
We know that $G_n(d)=p-lL_n(d)$. Therefore the proof follows from $(ii.)$.
\end{itemize}
\end{proof}

\begin{proof}[Proof of Proposition \ref{range n0}]
Let a fixed $D=(d_0,d_1, \ldots)$. From Lemma \ref{monotonicity} we know that $B_n(d)$ is decreasing in $d$ and $n$ which implies that: $\lim_{d \to \infty} B_n(d) \leq B_n(d_n) \leq B_n(0).$
From \eqref{SumBn} we obtain: $$B_n(0)=\frac{1-e^{-r(R-p)}\left( \frac{\mu}{\mu-r(c-l)}\right) ^{n+1}}{r} \text{ and } \lim_{d \to \infty} B_n(d)=\frac{1-e^{-r(R-p)}\left( \frac{\mu}{\mu-rc}\right) ^{n+1}}{r}.$$ 
Thus, $B_n(0)<0$ for all $n \geq \overline{n}$ and $ \lim_{d \to \infty} B_n(d)<0$ for all $n \geq \underline{n}$.

Since $B_n(d_n) \leq B_n(0)$ for all $n$, it follows that $B_n(d_n)<0$ for all $n \geq \overline{n}$. Therefore, $n_0(D) \leq \overline{n}$.
On the other hand, $B_n(d_n) \geq \lim_{d \to \infty} B_n(d)$ for all $n$, which implies that $B_n(d_n)\geq0$ for all $n < \underline{n}$. Thus, $n_0(D) \geq \underline{n}$.
Therefore, the range of  values of $n_0$ is: $\underline{n}\leq n_0(D)\leq\overline{n}.$
\end{proof}

\begin{proof}[Proof of Lemma \ref{dn_monotonicity}]
\begin{itemize}
\item[$(i.)$]
Since, from Lemma \ref{monotonicity}, $B_n(d)$ is decreasing in $d$, we have that $B_n(d) \geq 0 \text{ if and only if } d \leq \tilde{d}^{P}_{n}.$
For any $d \leq \tilde{d}^{P}_{n+1}$ it follows that $B_{n+1}(d) \geq 0$ and since $B_n(d)$ is decreasing in $n$ from Lemma \ref{monotonicity}, it follows $B_n(d) \geq 0,$ thus $d \leq \tilde{d}^{P}_{n}.$ We thus see that $\tilde{d}^{P}_{n+1} \leq \tilde{d}^{P}_{n}.$
\item[$(ii.)$]
For $n < \underline{n}, \: \lim_{d \to \infty} B_n(d) > 0,$ which implies that $\sup \lbrace d \geq 0: B_n(d) \geq 0\rbrace=\infty.$
Therefore, $\tilde{d}^{P}_{n}=\infty.$
For $n=\underline{n}, \: \lim_{d \to \infty} B_n(d) \geq 0.$
\item[$(iii.)$] We know from Lemma \ref{monotonicity} that $L_n(d)$ increasing in $n$ and decreasing in $d$. We have also proved that $\tilde{d}^{P}_{n+1} \leq \tilde{d}^{P}_{n}.$
Therefore, $L_n(\tilde{d}^{P}_{n}) \leq L_{n+1}(\tilde{d}^{P}_{n}) \leq L_{n+1}(\tilde{d}^{P}_{n+1}).$ Thus, $L_n(\tilde{d}^{P}_{n})$ is increasing in $n$ and $G_n(\tilde{d}^{P}_{n})$ is decreasing in $n$.
\end{itemize}
\end{proof}

\begin{proof}[Proof of Proposition \ref{opt quot}]
Let $n_0$ be fixed and $D_{n_0}=(d_0, d_1, d_2,\ldots,d_{n_0-1}, d_{n_0})$ be a feasible solution of \eqref{firstmaxpro}, with acceptance threshold $n_0,$ i.e., $B_n(d_n)\geq0$ for $n=0,1,\ldots,n_0-1$ and $B_{n_0}(d_{n_0})<0$ with $d_n\geq0$ for $n=0, 1,\ldots, n_0$.

Since $B_n(d_n)$ is decreasing from Lemma \ref{monotonicity}, it follows $d_n \leq \tilde{d}^{P}_{n}.$
Assume that $d_k < \tilde{d}^{P}_{k}$ for some $k \leq n_0-1.$ We will show that $D_{n_0}$ cannot be an optimal solution to \eqref{firstmaxpro}. To see this, define another solution $D^{'}_{n_0}=(d^{'}_0, d^{'}_1, d^{'}_2,\ldots, d^{'}_k,\ldots,d^{'}_{n_0-1}, d^{'}_{n_0})$ with $d^{'}_k = \tilde{d}^{P}_{k}$ and $d^{'}_n=d_n$  for all $n \neq k$.

Then $B_n(d^{'}_n)\geq0$ for $n=0,1,...,n_0-1$ and $B_{n_0}(d^{'}_{n_0})<0$ with $d^{'}_n \geq0$ for $n=0, 1,.., n_0$. Thus, $D^{'}_{n_0}$ is feasible. Also $F_k(d^{'}_k)>F_k(d_k)$, therefore $D_{n_0}$ is not optimal.
\end{proof}

\begin{proof}[Proof of Proposition \ref{optimalthresholdprovider}]

Letting $G_n=G_n(\tilde{d}^{P}_{n})$, $H(n_0)$ can be written as: 
$$H(n_0)=\lambda \frac{\displaystyle\sum_{n=0}^{n_0-1} \rho^{n} G_n}{\displaystyle\sum_{n=0}^{n_0} \rho^{n}}.$$ 
 
For any $\underline{n}\leq n_0\leq \overline{n}-1$ we have, 
\begin{align}\label{Aalgorithm}
H(n_0+1)-H(n_0)&=\lambda \frac{\displaystyle\sum_{n=0}^{n_0} \rho^{n} \displaystyle\sum_{n=0}^{n_0} \rho^{n} G_n -\displaystyle\sum_{n=0}^{n_0+1} \rho^{n} \displaystyle\sum_{n=0}^{n_0-1} \rho^{n} G_n }{\displaystyle\sum_{n=0}^{n_0+1} \rho^{n}\displaystyle\sum_{n=0}^{n_0} \rho^{n}} \nonumber \\
&=\lambda \frac{\rho^{n_0} G_{n_0} \displaystyle\sum_{n=0}^{n_0} \rho^{n} -\rho^{n_0+1}\displaystyle\sum_{n=0}^{n_0-1} \rho^{n} G_n}{\displaystyle\sum_{n=0}^{n_0+1} \rho^{n}\displaystyle\sum_{n=0}^{n_0} \rho^{n}} \nonumber\\
&=\lambda \frac{\rho^{n_0}}{\displaystyle\sum_{n=0}^{n_0+1} \rho^{n}\displaystyle\sum_{n=0}^{n_0} \rho^{n}} \left( G_{n_0} \displaystyle\sum_{n=0}^{n_0} \rho^{n}-\rho \displaystyle\sum_{n=0}^{n_0-1} \rho^{n} G_n \right) \nonumber\\
&=\lambda \frac{\rho^{n_0}}{\displaystyle\sum_{n=0}^{n_0+1} \rho^{n}\displaystyle\sum_{n=0}^{n_0} \rho^{n}}A(n_0).
\end{align}

We next show that $A(n_0)$ is decreasing in $n_0$. For any $\underline{n}\leq n_0\leq \overline{n}-2$,
\begin{align}\label{Amonoton}
A(n_0+1)-A(n_0)&=G_{n_0+1}\displaystyle\sum_{n=0}^{n_0+1} \rho^{n}-\rho \displaystyle\sum_{n=0}^{n_0} \rho^{n} G_n-G_{n_0}\displaystyle\sum_{n=0}^{n_0} \rho^{n}+\rho \displaystyle\sum_{n=0}^{n_0-1} \rho^{n} G_n \nonumber\\
&= \displaystyle\sum_{n=0}^{n_0} \rho^{n} \left( G_{n_0+1}-G_{n_0} \right)+\rho^{n_0+1} \left( G_{n_0+1}-G_{n_0} \right) \nonumber\\
&= \left( G_{n_0+1}-G_{n_0} \right) \displaystyle\sum_{n=0}^{n_0+1} \rho^{n} \nonumber\\
&= \left( G_{n_0+1}-G_{n_0} \right) \frac{1-\rho^{n_0+2}}{1-\rho} \leq 0,
\end{align}
since $G_n$ decreasing in $n$ from Lemma \ref{dn_monotonicity}.

From \eqref{Aalgorithm} and \eqref{Amonoton} it follows that
$H(n_0+1)-H(n_0) \geq 0$ for $n<\tilde{n}_P$ and $H(n_0+1)-H(n_0) < 0$ for $n \geq \tilde{n}_P$. 
Thus, $\tilde{n}_P$ is an optimal value that maximizes $G(n_0).$ 
\end{proof}

\begin{lemma}\label{Fn_monotonicitysoc}
For $n \leq n_0-1:$
\begin{itemize}
\item[(i.)] $(G_n+B_n)(\tilde{d}^{S}_{n})$ is decreasing in $n$
\item[(ii.)] $B_n(\tilde{d}^{P}_{n})\leq B_n(\tilde{d}^{S}_{n})$
\end{itemize}
\end{lemma}
\begin{proof}
\begin{itemize}
\item[$(i.)$] From Lemma \ref{monotonicity}, $(G_{n+1}+B_{n+1})(d) \leq (G_n+B_n)(d)$ for all $d \geq 0$.
From the definition of $\tilde{d}^{S}_{n}$, it implies that $(G_n+B_n)(d) \leq (G_n+B_n)(\tilde{d}^{S}_{n})$ for all $d \leq \tilde{d}^{P}_{n}$. 

Let $d=\tilde{d}^{S}_{n+1}$. From the above inequalities we have: $(G_{n+1}+B_{n+1})(\tilde{d}^{S}_{n+1}) \leq (G_n+B_n)(\tilde{d}^{S}_{n})$.
\item[$(ii.)$] It follows from lemma \ref{monotonicity} and the fact that $\tilde{d}^{P}_{n} \geq \tilde{d}^{S}_{n}$.
\end{itemize}
\end{proof}

\begin{proof}[Proof of Proposition \ref{optimalthresholdsoc}]

The proof is along similar lines as in Proposition \ref{optimalthresholdprovider}, using that $(G_n+B_n)(\tilde{d}^{S}_{n})$ decreasing in $n$ from Lemma \ref{Fn_monotonicitysoc}.
\end{proof}

\section*{Appendix B}
In Appendix B we provide some properties and theoretical results from probability theory that are necessary for the analysis in section \ref{Social_welfare}. Specifically, to determine the optimal single lead-time quotation strategy, we need to use the marginal distribution instead of the conditional one that we use in the dynamic problem. 

Let $n, X$ denote the number of customers that an arriving customer finds upon arrival and the sojourn time of this customer respectively, in a $M/M/1/n_0$ queue with service rate $\mu$ in steady-state.  

From the PASTA property the distribution of $n$ is identical to the steady-state distribution of the queue length, i.e.,
 \begin{equation*}
 q(n;n_0)=\frac{\rho^{n}(1-\rho)}{1-\rho^{n_0+1}}.
 \end{equation*}
The conditional distribution of $X$ given that there are $n$ customers in the queue is $Gamma(n+1,\mu)$. Let $f_{n,\mu}(x)$ and $\overline{F}_{n,\mu}(x)$ be the corresponding pdf and tail probability respectively. 

The marginal distribution of $X$ is determined as the sojourn time of a customer in a finite capacity queue. We obtain:
\begin{align*}
\displaystyle f_{\mu}(x)
&=\sum_{n=0}^{n_0-1} q(n;n_0) f_{n,\mu}(x)\\
&=\lambda\frac{(1-\rho)e^{-(\mu-\lambda)x}}{\rho(1-\rho^{n_0})} \sum_{n=0}^{n_0-1} \frac{(\lambda x)^{n} e^{-\lambda x}}{n!}\\
&=(\mu-\lambda ) e^{-(\mu-\lambda)x} \frac{Pois(n_0-1;\lambda x)}{1-\rho^{n_0}},
\end{align*}
and
\begin{equation*}
\overline{F}_{\mu}(x)=\int\limits_{x}^\infty (\mu-\lambda ) e^{-(\mu-\lambda)x} \frac{Pois(n_0-1;\lambda x)}{1-\rho^{n_0}}\mathrm{d}x,
\end{equation*}
where $Pois(n_0-1;\lambda x)$ the cdf of a Poisson distribution with rate $\lambda x$.

The corresponding hazard rated of the conditional and the marginal distribution of X are:
\begin{equation*}
h_{\mu}(x)=\frac{f_{\mu}(x)}{\overline{F}_{\mu}(x)},
\end{equation*}
\begin{equation*} 
h_{n,\mu}(x)=\frac{f_{n,\mu}(x)}{\overline{F}_{n,\mu}(x)}.
\end{equation*}

\subsection*{B.1. Dynamic lead-time quotation for social Optimization problem}
\begin{lemma}\label{lemmaforunimodal}
For $X\sim Gamma(n,\mu)$ and $Y\sim Gamma(n,v)$ with $n\in \mathbb{N}$ and $v<\mu$, it is true that:
\begin{itemize}
\item[(i.)] The hazard rate $h_{n,\mu}(d)$ is increasing in $\mu$
\item[(ii.)] $\frac{\overline{F}_{n,\mu}(d)}{\overline{F}_{n,v}(d)}$ is decreasing in $d$.
\end{itemize}
\end{lemma}
\begin{proof}
\begin{itemize}
\item[$(i.)$]
Using the relationship between Gamma and Poisson distribution it follows that: 
\begin{equation*}
\overline{F}_{n,\mu}(d)=P(X_1+\ldots+X_n>d)=Pois(n_0-1;\mu d)=\sum_{k=0}^{n-1} e^{-\mu d} \frac{(\mu d)^k}{k!}.
\end{equation*} 
As a result, 
\begin{equation*}
\frac{\overline{F}_{n,\mu}(d)}{f_{n,\mu}(d)}=\frac{e^{-\mu d} \sum_{k=0}^{n-1} \frac{(\mu d)^k}{k!} }{e^{-\mu d}\frac{\mu^{n}d^{n-1}}{(n-1)!}}=\frac{(n-1)!}{d^{n-1}} \sum_{k=0}^{n-1} \mu^{k-n}\frac{d^k}{k!}.
\end{equation*}
Since $k-n<0$, the last term is decreasing in $\mu$, thus $h_{n,\mu}(d)$ is increasing in $\mu$.

\item[$(ii.)$]
The first derivative of $\frac{\overline{F}_{n,\mu}(d)}{\overline{F}_{n,v}(d)}$ with respect to $d$ is:
\begin{equation*}
\left( \frac{\overline{F}_{n,\mu}(d)}{\overline{F}_{n,v}(d)}\right)^{'}=\frac{-f_{n,\mu}(d)\overline{F}_{n,v}(d)+f_{n,v}(d)\overline{F}_{n,\mu}(d)}{(\overline{F}_{n,v}(d))^2}.
\end{equation*}
From $(i.)$ it follows that $h_{n,v}(d)<h_{n,\mu}(d)$, therefore the numerate of the above fraction is negative and the proof is complete.
\end{itemize}
\end{proof}

\begin{proof}[Proof of Proposition \ref{unimodal}]
We fix $n \in \left\lbrace 0,1, \ldots ,\overline{n}-1 \right\rbrace$, and consider the problem of maximizing $G_n(d_n)+B_n(d_n)$ in $d_n\in[0,\tilde{d}^{P}_{n}]$.
The first derivative of $G_n(d_n)+B_n(d_n)$ is:
\begin{align*}
(G_n(d_n)+B_n(d_n))^{'}
&=l\int\limits_{d_n}^\infty \left( 1-e^{-r(R-p-(c-l)x-ld_n)} \right)   f_{n,\mu}(x)\mathrm{d}x\\
&=l\left(  \int\limits_{d_n}^\infty f_{n,\mu}(x)\mathrm{d}x -e^{-r(R-p)} e^{rld_n} \int\limits_{d_n}^\infty e^{r(c-l)x}   f_{n,\mu}(x)\mathrm{d}x \right). 
\end{align*}

Since, 
\begin{align*}
\int\limits_{d_n}^\infty e^{r(c-l)x}   f_{n,\mu}(x)\mathrm{d}x
&=\left( \frac{\mu}{v}\right)^{n+1} \int\limits_{d_n}^\infty \frac{(v)^{n+1}x^n e^{-(v)x}}{n!} \mathrm{d}x\\
&=\left( \frac{\mu}{v}\right)^{n+1} \int\limits_{d_n}^\infty  f_{n,v}(x)\mathrm{d}x\\
&=\left( \frac{\mu}{v}\right)^{n+1} \overline{F}_{n,v}(d_n),
\end{align*}
it follows that:
\begin{align}\label{a_d}
(G_n(d_n)+B_n(d_n))^{'}
&=l\left(  \overline{F}_{n,\mu}(d_n)-e^{-r(R-p)} e^{rld_n} \left( \frac{\mu}{v}\right)^{n+1} \overline{F}_{n,v}(d_n) \right)\nonumber \\
&=le^{rld_n}\overline{F}_{n,v}(d_n) \left( \frac{\overline{F}_{n,\mu}(d_n)}{\overline{F}_{n,v}(d_n)}e^{-rld}-\left( \frac{\mu}{v}\right)^{n+1} e^{-r(R-p)} \right)\nonumber\\
&=le^{rld_n}\overline{F}_{n,v}(d_n) a(d_n),
\end{align}
where $a(d_n)$ is defined in \eqref{dtildesoc}. We will show that the derivative in \eqref{dtildesoc} is either minimized at a unique point in $[0,\tilde{d}^{P}_{n}]$ or it is always positive. From Lemma \ref{lemmaforunimodal} it follows that $a(d_n)$ is decreasing in $d_n$.

For $d_n=0$, we have $\overline{F}_{n,\mu}(0)=\overline{F}_{n,v}(0)=1$. Thus:
\begin{equation*}
a(0)=1-e^{-r(R-p)} \left(\frac{\mu}{v}\right)^{n+1} \geq0.
\end{equation*} 

For $d_n \to \infty$, by applying de L' Hospital's rule,
\begin{align*}
\lim_{d_n \to \infty} \frac{\overline{F}_{n,\mu}(d_n)e^{-rld_n}}{\overline{F}_{n,v}(d_n)}
&=\lim_{d_n \to \infty} \frac{f_{n,\mu}(d_n)e^{-rld_n}+rle^{-rld_n}\overline{F}_{n,\mu}(d_n)}{f_{n,v}(d_n)}\\
&=\lim_{d_n \to \infty} \frac{f_{n,\mu}(d_n)e^{-rld_n}}{f_{n,v}(d_n)} \left( 1+\frac{1}{h_{n,\mu}(d_n)} \right)     
\end{align*}
and also,
\begin{equation*}
\lim_{d_n \to \infty} \frac{f_{n,\mu}(d_n)e^{-rld_n}}{f_{n,v}(d_n)}= \left(\frac{\mu}{v}\right)^{n+1} e^{-rcd_n}=0,
\end{equation*}
and
\begin{equation*}
\lim_{d_n \to \infty} \frac{1}{h_{n,\mu}(d_n)}>0,
\end{equation*}
since the hazard rate $h_{n,\mu}(d_n)$ of Gamma distribution is increasing in $d_n$.

It follows that:
\begin{equation*}
\lim_{d_n \to \infty} a(d_n)=-\left( \frac{\mu}{v}\right)^{n+1} e^{-r(R-p)}<0.
\end{equation*}
Since, $a(0) \geq 0$ and $\lim_{d_n \to \infty} a(d_n)<0$ and $a(d_n)$ is decreasing in $d_n$, there is a unique $d^0<\infty$ such that $a(d^0)=0$. Furthermore, $a(d_n) \geq 0$ for all $d_n \leq d^0$ and $a(d_n) < 0$ otherwise. 

Returning to \eqref{a_d}, we consider the following cases regarding $\tilde{d}^{S}_{n}$:
\begin{itemize}
\item[$(i.)$] If $\tilde{d}^{P}_{n}=\infty$, then from the previous discussion on $a(d_n)$, it follows that $(G_n(d_n)+B_n(d_n))^{'} \geq 0$ for all $d_n \leq d^0$ and negative for $d_n > d^0$.
Therefore, the maximum of $G_n(d_n)+B_n(d_n)$ in $d_n\in[0,\infty]$ is $\tilde{d}^{S}_{n}=d^0.$
\item[$(ii.)$] If $\tilde{d}^{P}_{n}<\infty$, then the maximum point of $G_n(d_n)+B_n(d_n)$ depends on the sign of $a(\tilde{d}^{P}_{n})$. If $a(\tilde{d}^{P}_{n})<0$ then $\tilde{d}^{S}_{n}=d^0\in[0,\tilde{d}^{P}_{n}]$, otherwise $\tilde{d}^{S}_{n}=\tilde{d}^{P}_{n}$.
\end{itemize}
\end{proof}

\subsection*{B.2. Single lead-time quotation for social Optimization problem}

\begin{lemma}\label{lemmaforunimodalsingle}
Let $X$, $Y$ be the sojourn times of a customer in two finite capacity $M/M/1/n_0$ queues in steady state, with service rates $\mu$ and $v$, respectively, and $v<\mu$. Then:
\begin{itemize}
\item[(i.)] The hazard rate $h_{\mu}(d)$ is increasing in $\mu$
\item[(ii.)] $\frac{\overline{F}_{\mu}(d)}{\overline{F}_{v}(d)}$ is decreasing in $d$.
\end{itemize}
\end{lemma}
\begin{proof}
\begin{itemize}
\item[$(i.)$] It is immediate that $h_{\mu}(d)$ is increasing in $\mu$ since:
\begin{equation*}
h_{\mu}(d)=\frac{Pois(n_0-1;\mu d)}{\int\limits_{d}^\infty e^{(\mu-\lambda)(d-x)} \frac{Pois(n_0-1;\mu d)}{1-\rho^{n_0}}\mathrm{d}x}.
\end{equation*}

\item[$(ii.)$]
The proof follows by differentiation of $\frac{\overline{F}_{\mu}(d)}{\overline{F}_{v}(d)}$ as in the proof of Lemma \ref{lemmaforunimodal}, $(ii.)$.
\end{itemize}
\end{proof}

\begin{proof}[Proof of Proposition \ref{unimodalc}]

Using Lemma \ref{optsocc}, we consider three cases regarding $n_0$.
\begin{itemize}
\item[$(i.)$]
We first fix $n_0 \in \left\lbrace \underline{n}+1, \ldots ,\overline{n}-1 \right\rbrace$, and consider the problem of maximizing $S_c(d)$, for $d\in (\tilde{d}^{P}_{n_0},\tilde{d}^{P}_{n_0-1}]$ where:
$$S_c(d)=\lambda \sum_{n=0}^{n_0-1} q(n;n_0)(G_n(d)+B_n(d)).$$
In the proof of Proposition \ref{unimodal}, it was shown that:
\begin{equation*}
\left(G_n(d)+B_n(d)\right)^{'}=l\left(  \overline{F}_{n,\mu}(d)-e^{-r(R-p)} e^{rld} \left( \frac{\mu}{v}\right)^{n+1} \overline{F}_{n,v}(d) \right)
\end{equation*}
Hence, 
\begin{align*}
\frac{\partial S_c(d)}{\partial d}
&=\lambda l \left(  \sum_{n=0}^{n_0-1} q(n;n_0)\overline{F}_{n,\mu}(d)- e^{-r(R-p)} e^{rld} \sum_{n=0}^{n_0-1} q(n;n_0) \left( \frac{\mu}{v}\right)^{n+1} \overline{F}_{n,v}(d)\right)\\ 
\end{align*}
We have that $$\sum_{n=0}^{n_0-1} q(n;n_0)\overline{F}_{n,\mu}(d)=\overline{F}_{\mu}(d),$$
and, 
\begin{align*}
\sum_{n=0}^{n_0-1} q(n;n_0) \left( \frac{\mu}{v}\right)^{n+1} \overline{F}_{n,v}(d)
&=\sum_{n=0}^{n_0-1} \frac{(1-\rho) \rho^n}{1-\rho^{n_0+1}} \left(\frac{\mu}{v}\right)^{n+1} \overline{F}_{n,v}(d)\\
&=\left(\frac{\mu-\lambda}{1-\rho^{n_0+1}}\right) \left(\frac{1-\left( \frac{\lambda}{v}\right)^{n_0+1}}{v-\lambda}\right) \sum_{n=0}^{n_0-1} \frac{\left( \frac{\lambda}{v}\right)^{n} \left( 1-\frac{\lambda}{v}\right) }{1-\left( \frac{\lambda}{v}\right)^{n_0+1}} \overline{F}_{n,v}(d)\\
&=\beta \overline{F}_{v}(d),
\end{align*}
with $\beta>0$.

Thus, the first derivative of $S_c(d)$ with respect to $d$ is equal to:
\begin{align}\label{a_c(d)}
\frac{\partial S_c(d)}{\partial d}
&=\lambda l \overline{F}_{v}(d) e^{rld} \left( \frac{\overline{F}_{\mu}(d)}{\overline{F}_{v}(d)} e^{-rld}-\beta e^{-r(R-p)} \right)\nonumber \\
&=\lambda l \overline{F}_{v}(d) e^{rld} a_c(d),
\end{align}
where $a_c(d)$ is defined in \eqref{dtildesocc}. From Lemma \ref{lemmaforunimodalsingle} it follows that $a_c(d)$ is decreasing in $d$.

Therefore, either $a_c(d) \geq 0$ for all $d\in(\tilde{d}^{P}_{n_0},\tilde{d}^{P}_{n_0-1}]$, in which case ${d}^{S_c}_{n}=\tilde{d}^{P}_{n_0-1}$, 
or $a_c(d) \leq 0$ for all $d\in(\tilde{d}^{P}_{n_0},\tilde{d}^{P}_{n_0-1}]$,  in which case there exist  $\epsilon$-optimal lead-time quotes which approach the sup in \eqref{firstsocproc} arbitrary close, since $a_c(d)$ is continuous and decreasing and the lower bound cannot be attained by any $d$,
or the maximizing value ${d}^{S_c}_{n_0}\in (\tilde{d}^{P}_{n_0},\tilde{d}^{P}_{n_0-1}]$ is the unique solution of $a_c(d)=0$.

\item[$(ii.)$] When $n_0=\underline{n}$, we maximize $S_c(d)$ for $d\in (\tilde{d}^{P}_{n_0},\infty]$.

For $d \to \infty$, by applying de L'Hospital's rule, 
\begin{align*}
\lim_{d \to \infty} \frac{\overline{F}_{\mu}(d)}{\overline{F}_{v}(d)}e^{-rld}
&=\lim_{d \to \infty} \frac{f_{\mu}(d)e^{-rld}}{f_{v}(d)}+rl\lim_{d \to \infty} \frac{e^{-rld}\overline{F}_{\mu}(d)}{f_{v}(d)}\\
&=\lim_{d \to \infty} \frac{f_{\mu}(d)e^{-rld}}{f_{v}(d)} \left( 1+\frac{1}{h_{\mu}(d)} \right) 
\end{align*}
and also,
\begin{equation*}
\lim_{d \to \infty} \frac{f_{\mu}(d)e^{-rld}}{f_{v}(d)}= \beta e^{-rcd}=0,
\end{equation*}
and,
\begin{equation*}
\lim_{d \to \infty} \frac{1}{h_{\mu}(d)}>0,
\end{equation*}
since the hazard rate $h_{\mu}(d)$ is increasing in $d$.

It follows that:
\begin{equation*}
\lim_{d \to \infty} a_c(d)=-\beta e^{-r(R-p)}<0.
\end{equation*}

Since $\lim_{d \to \infty} a_c(d)<0$ and $a_c(d)$ is decreasing in $d$, returning to \eqref{a_c(d)}, the maximum point of $G_n(d)+B_n(d)$ depends on the sign of $a_c(\tilde{d}^{P}_{\underline{n}})$. If $a_c(\tilde{d}^{P}_{\underline{n}}) >0$, then  ${d}^{S_c}_{\underline{n}}$ is the unique solution of $a_c(d)=0$. Otherwise, there exist $\epsilon$-optimal lead-time quotes which approach the sup in \eqref{firstsocproc} arbitrary close.

\item[$(iii.)$] When $n_0=\overline{n}$, we maximize $S_c(d)$ for $d\in [0,\tilde{d}^{P}_{n_0-1}]$.

For $d=0$, we have $\overline{F}_{\mu}(0)=\overline{F}_{v}(0)=1$. Thus:
\begin{equation*}
a_c(0)=1- \beta e^{-r(R-p)} ,
\end{equation*} 
where $a_c(0)$ can be either nonpositive or nonnegative.

In this case there is the closed interval $[0, \tilde{d}^{P}_{\overline{n}}]$. The only difference with the approach of case $(i.)$ is that when $a_c(d) \leq 0$ for all $d\in[0, \tilde{d}^{P}_{\overline{n}}]$, the optimal quote can be attained and it is ${d}^{S_c}_{\overline{n}}=0$.
\end{itemize}
\end{proof}

\begin{proof}[Proof of Lemma \ref{noepsilon}]

Let $n_0$ be a balking threshold with $a_c(\tilde{d}^{P}_{n_0})\leq 0$. We know that $Z_{n_0}(d)$ is decreasing in $n_0$ and that $a_c(d)$ is decreasing in $d$. Moreover the sign of $Z^{'}_{n_0}(d)$ is determined by the sign of $a_c(d)$.
Therefore $Z^{'}_{n_0}(d)\leq 0$ and $Z_{n_0}(\tilde{d}^{P}_{n_0})\geq Z_{n_0}(d)$ for any $d\in(\tilde{d}^{P}_{n_0},\tilde{d}^{P}_{n_0-1}]$.

However, $\tilde{d}^{P}_{n_0}$ is not a feasible value when the balking threshold is $n_0$ but it is feasible when the balking threshold is $n_0-1$. Additionally, $Z_{n_0-1}(\tilde{d}^{P}_{n_0})\geq Z_{n_0}(\tilde{d}^{P}_{n_0})$.

From the above, it follows that $Z_{n_0-1}(\tilde{d}^{P}_{n_0})\geq Z_{n_0}(d)$ for any $d\in[\tilde{d}^{P}_{n_0},\tilde{d}^{P}_{n_0-1}]$ thus $n_0$ cannot be an optimal balking threshold.
\end{proof}

\end{document}